\documentclass[11pt]{article}

\usepackage[utf8]{inputenc}
\usepackage{amsmath,mathtools} 
\usepackage{amssymb} 
\usepackage{graphicx} 
\usepackage{enumerate}
\usepackage{array}
\usepackage{anysize}
\usepackage[english]{babel}
\usepackage{textcomp}
\usepackage{makeidx}
\usepackage{longtable}
\usepackage{amsthm}
\usepackage{color}
\usepackage{hyperref}
\usepackage{caption}
\usepackage{subcaption}
\usepackage{tikz}
\usepackage{bbm}
\usepackage{diagbox}
\usepackage{authblk}

\newcommand{\besov}{B^{-d/2}_{\infty,\infty}}
\newcommand{\Besov}{B^{-d/2}_{\infty,\infty}(\mathbb{T}^d)}
\newcommand{\BV}{BV(\mathbb{T}^d)}
\newcommand{\bv}{BV}

\newcommand{\Td}{\mathbb{T}^d}
\newcommand{\Rd}{\mathbb{R}^d}
\newcommand{\Zd}{\mathbb{Z}^d}

\theoremstyle{definition}
\newtheorem{definition}{Definition}
\newtheorem{theorem}{Theorem}
\newtheorem{proposition}{Proposition}
\newtheorem{lemma}{Lemma}

\newtheorem{remark}{Remark}

\newtheorem{assumption}{Assumption}

\newtheorem{example}{Example}
\newtheorem*{examples}{Examples}
\newtheorem*{otherex}{Other examples}
\newtheorem*{Maintheorem}{Main Theorem}

\numberwithin{equation}{section}

\usepackage{natbib}

\title{Frame-constrained Total Variation Regularization for White Noise Regression}
\author[1]{Miguel~del~\'Alamo}
\author[1]{Housen~Li}
\author[1,2]{Axel~Munk}

\affil[1]{Institute for Mathematical Stochastics, University of G\"ottingen \\ Goldschmidtstrasse 7, 37077 G\"ottingen, Germany}
\affil[2]{Max Planck Institute for Biophysical Chemistry, Am Fassberg 11, 37077 G\"ottingen, Germany}
\date{\today}

\begin{document}

\maketitle

\begin{abstract}
Despite the popularity and practical success of total variation (TV) regularization for function estimation, surprisingly little is known about its theoretical performance in a statistical setting. While TV regularization has been known for quite some time to be minimax optimal for denoising one-dimensional signals, for higher dimensions this remains elusive until today. In this paper we consider frame-constrained TV estimators including many well-known (overcomplete) frames in a white noise regression model, and prove their minimax optimality w.r.t.~$L^q$-risk ($1\leq q<\infty$) up to a logarithmic factor in any dimension $d\geq 1$. Overcomplete frames are an established tool in mathematical imaging and signal recovery, and their combination with TV regularization has been shown to give excellent results in practice, which our theory now confirms. Our results rely on a novel connection between frame-constraints and certain Besov norms, and on an interpolation inequality to relate them to the risk functional. 
\end{abstract}

\begin{center}
\end{center}

\begin{center}
\end{center}

\tableofcontents

\section{Introduction}

We consider the problem of estimating a real-valued function $f$ from observations in the commonly used Gaussian white noise regression model (see e.g.~\cite{brown1996asymptotic},~\cite{reiss2008} and~\cite{tsybakov2009introduction})
\begin{equation}
 dY(x)=f(x)\, dx+\frac{\sigma}{\sqrt{n}}\,dW(x),\ \ x\in[0,1)^d.\label{Intro:WNmodel}
\end{equation}
Here, $dW$ denotes the standard Gaussian white noise process in $L^2(\Td)$, and we identify the $d$-torus $\Td\sim \Rd/\mathbb{Z}^d$ with the set 
$[0,1)^d$, i.e.~to simplify technicalities we assume $f$ to be a $1$-periodic function (see Remark~\ref{Res:RemPeriodic} in Section~\ref{Sect:Res} for the arguments to treat the nonperiodic case). To ease notation we will henceforth drop the symbol $\Td$, and write for instance $L^2$ instead of $L^2(\Td)$, and so on. The function $f$ is assumed to be of bounded variation ($\bv$), written $f\in\bv$, meaning that $f\in L^1$ and its weak partial 
derivatives of first order are finite Radon measures on $\Td$ (see Section~\ref{Res:Basic_Def} or~Chapter 5 in~\cite{Evans}). Note that, for~\eqref{Intro:WNmodel} to be well-defined, we need 
to assume additionally that $f\in L^2$ if $d\geq 3$, since only in $d=1,2$ we have $f\in \bv\subset L^2$. 
In the following we assume that $\sigma$ is known, otherwise it 
can be estimated $\sqrt{n}$-efficiently (see e.g.~\cite{munk2005difference} or~\cite{spokoiny2002variance}), which will not affect our results. 
In the following we use the terms bounded variation (BV) and total variation (TV) indistinctly. The former is commonly used in analysis, while the latter appears in imaging.

Functions of bounded variation can have discontinuities, and are thus ideal to model objects with edges and abrupt changes.
This is a desirable property for instance in medical imaging applications, where sharp transitions 
between tissues occur, and smoother functions would represent them inadequately (see e.g.~\cite{LiMRI} for a TV-based optical flow method in real time magnetic resonance imaging or~\cite{JiangPA} for its use in photoacoustic tomography). 
Consequently, $BV$ functions have been studied extensively in the applied and computational analysis literature, see e.g.~\cite{chambolle1997image},~\cite{Meyer},~\cite{ROF},~\cite{scherzer2009variational} and references therein.
Remarkably, the very reason for the success of functions of bounded variation in applications, namely their low smoothness, 
has hindered the development of a rigorous theory for the corresponding estimators in a statistical setting. 
With the exception of the one-dimensional case $d=1$, where total variation (TV) penalized least squares~\citep{mammen1997} and wavelet thresholding~\citep{donoho1998minimax} applied to $BV$ functions are known to attain the minimax optimal convergence rate $O(n^{-1/3})$, there are to the best of our knowledge no statistical guarantees for estimating $BV$ functions in dimension $d\geq 2$. 
Roughly speaking, the main challenges in higher dimensions are twofold: first, the embedding $BV\hookrightarrow L^{\infty}$ \textit{fails} if $d\geq 2$; and second, the space $BV$ does not admit a characterization in terms of the size of wavelet coefficients. More generally, $BV$ does not admit an unconditional basis (see Sections 17 and 18 in~\cite{Meyer}).

Our goal in this paper is to fill that gap. We consider the continuous model~\eqref{Intro:WNmodel} and present estimators for $f\in BV$ that are minimax optimal up to logarithmic factors in any dimension, i.e.~they attain the polynomial rate $n^{-1/(d+2)}$ for the $L^q$-risk, $q\in[1,1+2/d]$, and the rate $n^{-1/dq}$ for $q\in[1+2/d,\infty)$. 
While the first regime is well-known (e.g.~for $d=1$ and $q=2$, see again~\cite{mammen1997} and~\cite{donoho1998minimax}), much less attention has been paid to the second regime. We mention~\cite{goldenschluger} and~\cite{lepskii2015} for estimation over anisotropic Nikolskii classes, which in the isotropic case coincide with Besov spaces $B^s_{p,\infty}$, and~\cite{sadhanala2016total} for the case of discrete total variation when $q=2$ (see "Related work" later in this section for a comprehensive discussion). These risk regimes explain the recently observed phase transitions in discrete TV-regularization~\citep{sadhanala2016total} and component-wise isotone estimation~\cite{han2017isotonic} (see Figure~\ref{fig:Exponent} and the remarks after the Main Theorem in the Introduction for more details). As a remarkable statistical consequence we also show that there is no $L^{\infty}$-consistent estimator of $BV$ functions.

\begin{figure}
\hspace{3cm}
\includegraphics[width=.65\linewidth,trim={{0.3\linewidth} {1.2\linewidth} {0.4\linewidth} {0.2\linewidth}},clip]{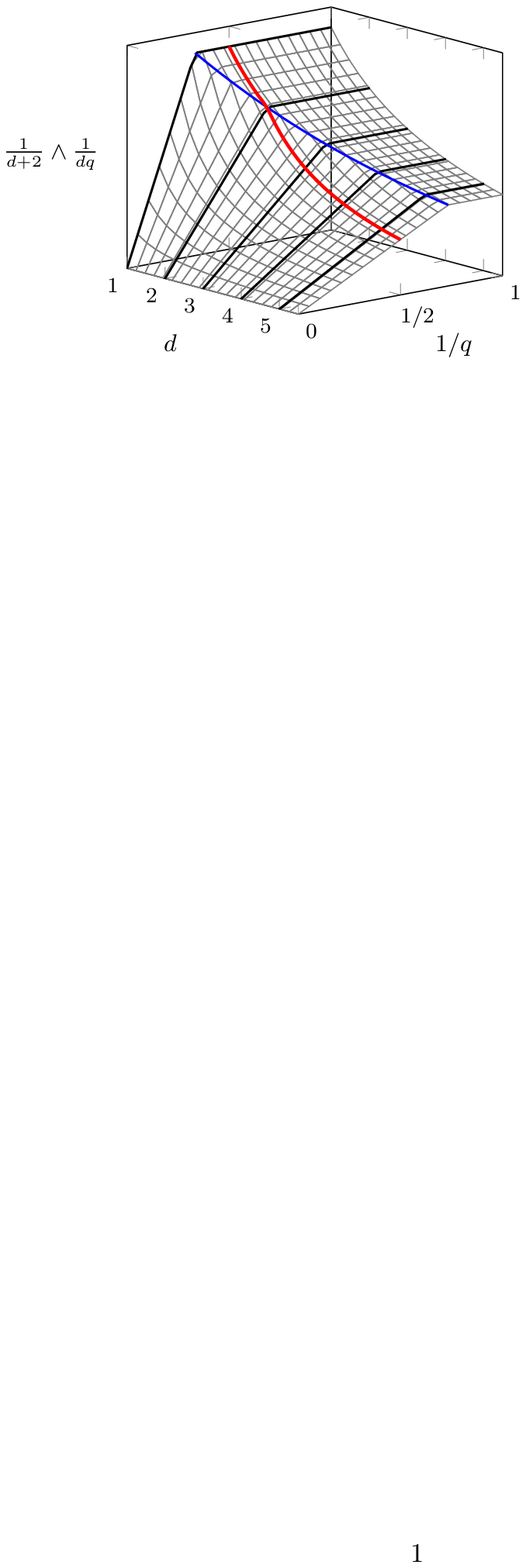}
\vspace{-1cm}
\caption{Exponent of the minimax rate over $BV_L$, $\min\{\frac{1}{d+2},\frac{1}{dq}\}$, plotted as a function of $d\in\mathbb{N}$ and $1/q\in[0,1]$. The line $1/q=d/(d+2)$ is marked in blue, and the red line corresponds to the $L^2$-risk, $q=2$. The phase transition observed in~\cite{sadhanala2016total} for the $L^2$-minimax risk corresponds to the change of behavior of the red curve.}
\label{fig:Exponent}
\end{figure}

The estimators that achieve these rates are not a straightforward extension of those for $d=1$~\citep{mammen1997}. There it is sufficient to penalize a \textit{global} least-squares data-fidelity term by the TV functional, i.e.,
\begin{equation}
\hat{f}_{\lambda_n}\in\underset{g}{\textup{argmin}}\, \|g-Y\|_2^2+\lambda_n |g|_{BV}\label{Intro:ROF}
\end{equation}
for a suitable sequence of Lagrange multipliers $\lambda_n$, where $|g|_{BV}$ denotes the $BV$-seminorm of $g$ (Section~\ref{Res:Basic_Def}). 
Instead, we consider estimators that combine the strengths of TV and \textit{multiscale} data-fidelity constraints. Multiscale data-fidelity terms and the associated reconstructions by the corresponding dictionary are widely used since the introduction of wavelets (see e.g.~\cite{daubechies1992ten} and~\cite{donoho1993unconditional}), and specially for imaging tasks overcomplete frames such as curvelets~\citep{candes2000curvelets}, shearlets (\cite{guo2006sparse},~\cite{labate2005sparse}) and other multiresolution systems (see~\cite{haltmeier2014extreme} for a survey) have been shown to perform well in theory and numerical applications. 
In contrast, for the multiscale TV-estimators a theoretical understanding in a statistical setup when $d\geq 2$ is lacking, although its good empirical performance has been reported for specific choices of dictionaries in several places~\cite{CandesGuo}~\cite{dong2011automated}~\cite{frick2012}~\cite{frick2013statistical}, see also Figure~\ref{fig:Intro1}.
Further, these methods were rarely used in routine applications, as they need large scale nonsmooth convex optimization methods for their computation. However, in the meantime such methods have become computationally feasible due to recent progress in optimization, e.g.~the development of primal-dual algorithms~\citep{chambolle} or semismooth Newton methods~\citep{clason2010semismooth}. 
Hence, we do see practical potential for such multiscale TV-methods, for which we give a theoretical justification in this paper in large generality.

\subsection*{Multiscale total variation estimators}

Let $\Phi=\big\{\phi_{\omega}\, \big|\, \omega\in\Omega\big\}\subset L^2$ be a dictionary of functions indexed by a countable set 
$\Omega$ and satisfying $\|\phi_{\omega}\|_{L^2}=1$, $\omega\in\Omega$. Consider the projection of the white noise model~\eqref{Intro:WNmodel} onto $\Phi$,
\begin{equation}
 Y_{\omega}:=\langle \phi_{\omega},f\rangle+\frac{\sigma}{\sqrt{n}}\int_{\Td}\phi_{\omega}(x)\, dW(x), \ \ \omega\in\Omega,\label{Intro:Ycoeff}
\end{equation}
where $\langle\cdot,\cdot\rangle$ denotes the standard inner product in $L^2$. For each $n\in\mathbb{N}$, $\Phi$ and given the 
observations $Y_{\omega}$, our estimator $\hat{f}_{\Phi}$ for $f$ is defined as any solution to the constrained minimization problem
\begin{equation}
 \hat{f}_{\Phi}\in\underset{g\in X_n}{\textup{ argmin }}|g|_{\bv}\ \textup{ subject to } \ \max_{\omega\in\Omega_n}\big|\langle \phi_{\omega},g\rangle-Y_{\omega}\big|\leq\gamma_n.\label{Intro:TV_est1}
\end{equation}
Here, $X_n\subset BV$ is a suitable closed, convex set which may depend on $n$ (see~\eqref{Res:Aux_Set} for the definition). 
Hence, the existence of a minimizer is guaranteed by the convexity and lower-semicontinuity of the objective function and the constraint. 
The \textit{finite} subsets $\Omega_n\subset\Omega$ indexing a proper sequence of subsets of the dictionary $\Phi$ will be specified later (see Assumption~\ref{Res:Gen_Ass} and~\eqref{Intro:BesovCompat} below). For instance, if $\Phi$ is a wavelet basis, $\Omega_n$ corresponds to the wavelet coefficients at all scales $j$ such that $2^{jd}\leq n$. 

The constraint in~\eqref{Intro:TV_est1} can be interpreted statistically as testing whether the data $Y_{\omega}$ is compatible with the coefficients $\langle \phi_{\omega},\hat{f}_{\Phi}\rangle$, \textit{simultaneously} for all $\omega\in\Omega_n$, an approach that dates back to~\cite{nemirovskii1985nonparametric}. This testing interpretation suggests how to choose the parameter $\gamma_n$ in~\eqref{Intro:TV_est1}: the coefficients $\langle \phi_{\omega},f\rangle$ of the truth should satisfy the constraint with high probability. This can be achieved by the \textit{universal threshold} 
\begin{equation}
\gamma_n(\kappa)=\kappa\sigma\sqrt{\frac{2\log \#\Omega_n}{n}} \ \ \textup{ for } \ \ \kappa>\kappa*\label{Res:UnivGamma}
\end{equation}
with $\kappa^*>0$ depending on $d$ and the dictionary $\Phi$ in an explicit way (see Theorem~\ref{Res:Main_thm}). This universal choice of the parameter $\gamma_n$ appears to us as a great conceptual and practical advantage of the estimator~\eqref{Intro:TV_est1}, in contrast to its penalized formulation, requiring more complex parameter-choice methods (e.g.~\cite{lepskii1991problem} or~\cite{wahba1977practical}). In particular, $\gamma_n$ in~\eqref{Res:UnivGamma} can be precomputed using known or simulated quantities only.

The main conceptual contribution of this paper is to link the multiscale constraint in~\eqref{Intro:TV_est1} and the Besov $\besov$ norm. In fact, several dictionaries $\Phi$ used in practice have the following property: for each $n\in\mathbb{N}$ there is a finite subset $\Omega_n\subset \Omega$ such that
\begin{equation}
\|g\|_{\besov}\leq C\max_{\omega\in\Omega_n}\big|\langle \phi_{\omega},g\rangle\big|+C\frac{\|g\|_{L^{\infty}}}{\sqrt{n}}\label{Intro:BesovCompat}
\end{equation}
holds for any function $g\in L^{\infty}$. This is a Jackson-type inequality~\citep{cohen2003numerical}, representing how well a function can be approximated in the Besov $\besov$ norm by its coefficients with respect to $\Phi$. It is well-known that smooth enough wavelet bases satisfy this condition~\citep{cohen2003numerical}. In Section~\ref{Sect:Examps} we will show that~\eqref{Intro:BesovCompat} holds for more general multiscale systems, e.g.~systems of indicator functions of dyadic cubes, and mixed frames of wavelets and curvelets and of wavelets and shearlets. In practice, the inequality~\eqref{Intro:BesovCompat} allows us to relate the statistical multiscale constraint in~\eqref{Intro:TV_est1} to an analytic object (the Besov norm). With this connexion, we leverage tools from harmonic analysis to analyze the performance of the estimator~\eqref{Intro:TV_est1}.

 For fixed $L>0$, define the $BV\cap L^{\infty}$-ball of radius $L$,
\begin{align}
 BV_L&:=\big\{g\in BV\cap L^{\infty} \, \big|\, \|g\|_{L^{\infty}}\leq L, \ |g|_{\bv}\leq L\big\}.\label{Intro:Param1}
\end{align}

The main contribution of this paper (Theorems~\ref{Res:Main_thm} and~\ref{MinimaxThm} in Section~\ref{Sec:Main}) can be informally stated as follows.
\begin{Maintheorem}[Informal]
Let the dimension $d\in\mathbb{N}$, and let $\Phi$ satisfy an inequality of the form~\eqref{Intro:BesovCompat} (see Assumption~\ref{Res:Gen_Ass} in Section~\ref{Sec:Main}). Let the threshold $\gamma_n$ in~\eqref{Intro:TV_est1} be as in~\eqref{Res:UnivGamma}. Then the estimator $\hat{f}_{\Phi}$ in~\eqref{Intro:TV_est1} attains the \textit{minimax optimal} rate of convergence over $BV_L$ possibly up to a logarithmic factor ($(\log n)^2$ in $d=1$ and $\log n$ else)
\begin{equation}
     \sup_{f\in BV_L}\mathbb{E}\big[\|\hat{f}_{\Phi}-f\|_{L^q}\big]\leq C_L\, n^{-\min\{\frac{1}{d+2},\frac{1}{dq}\}}\label{Intro:Conver1}
\end{equation}
for $n$ large enough, for any $q\in\big[1,\infty\big)$, any $L>0$ and a constant $C_L>0$ independent of $n$ and $q$, but dependent on $L$, $\sigma$, $d$ and $\Phi$. .
\end{Maintheorem}

We remark that this reproduces the results by~\cite{sadhanala2016total} for estimating $BV$ functions in a discrete model for $q=2$ (quadratic risk). Indeed,~\cite{sadhanala2016total} shows that the minimax rate with respect to the empirical $\ell^2$-risk scales as 
$n^{-\min\{\frac{1}{d+2},\frac{1}{2d}\}}$. Our theorem explains this "phase transition" in the risk between $d\leq 2$ and $d>2$ as arising from the low smoothness of $BV$ functions and from the $L^q$-risk employed (see Figure~\ref{fig:Exponent} for an illustration of this).

Notably, the minimax rate in the Main Theorem for $q=2$ also matches the minimax rate derived in~\cite{han2017isotonic} for estimating bounded, component-wise isotone functions in a discretized setting with respect to the empirical $\ell^2$-risk. Remarkably, this means that the statistical complexity of estimating $BV$ functions equals the complexity of estimating component-wise isotone functions, arguably a much simpler class. This result is well-known in dimension $d=1$, as a function of bounded variation can be written as the difference of two monotone functions, but we are not aware of any such result in $d\geq 2$. 
Moreover, this complements the recent finding that entirely monotone functions have the same statistical complexity as functions of bounded variation in the sense of Hardy-Krause~\cite{fang2019}. We remark, however, that bounded variation in the sense of Hardy-Krause is a much stronger assumption than bounded variation in the sense that we use here (see "Related work" for a discussion).

The proof of~\eqref{Intro:Conver1} relies on the compatibility between the frame constraint and the $\besov$ norm, as expressed in~\eqref{Intro:BesovCompat}. This allows us to use techniques from harmonic analysis to analyze $\hat{f}_{\Phi}$, such as the interpolation inequality between $\besov$ and $\bv$~\citep{cohen2003harmonic},
\begin{equation}
   \|g\|_{L^q}\leq C \|g\|_{B^{-d/2}_{\infty,\infty}}^{\frac{2}{d+2}}\|g\|_{BV}^{\frac{d}{d+2}}\hspace{0.5cm} \forall g\in\besov\cap BV\label{Intro:DummyInt}
\end{equation}
for any $q\in\big[1,\frac{d+2}{d}\big]$, $d\geq 2$. 
This interpolation inequality relates the risk functional on the left-hand side with the data-fidelity and the regularization functionals on the right-hand side. 
It can be proven by a delicate analysis of the wavelet coefficients of functions of bounded variation (the original proof is in~\cite{cohen2003harmonic}, and here we use an extension of~\eqref{Intro:DummyInt} to periodic functions). The inequality~\eqref{Intro:DummyInt} is the first step towards bounding the $L^q$-risk of 
$\hat{f}_{\Phi}$: inserting $g=\hat{f}_{\Phi}-f$ we can bound it in terms of the $\besov$ and the $BV$-risks. 
It can be shown that the $BV$-risk is bounded by a constant with high probability, while the $\besov$-risk can be handled using inequality~\eqref{Intro:BesovCompat} as 
follows:
\begin{align}
 \|\hat{f}_{\Phi}-f\|_{\besov}&\leq  C\max_{\omega\in\Omega_n}\big|\langle\phi_{\omega},\hat{f}_{\Phi}\rangle-Y_{\omega}\big|+C\frac{\sigma}{\sqrt{n}}\max_{\omega\in\Omega_n}\bigg|\int \phi_{\omega}(x)\, dW(x)\bigg|\nonumber
 \\
 &\hspace{0.5cm}+C\frac{\|\hat{f}_{\Phi}-f\|_{L^{\infty}}}{\sqrt{n}}.\label{Intro:Error1}
 \end{align}
The first term is bounded by $\gamma_n=O(n^{-1/2}\sqrt{\log\#\Omega_n})$ as in~\eqref{Res:UnivGamma} by construction, and it represents the error that we 
allow the minimization procedure to make. The second term behaves as $O(n^{-1/2}\sqrt{\log\#\Omega_n})$ asymptotically almost surely, 
and it represents the stochastic error of the estimator. The third term arises from the compatibility between $\Phi$ and the Besov space $\besov$ stated in~\eqref{Intro:BesovCompat}. 
Inserting the result in~\eqref{Intro:DummyInt} (which requires $d\geq 2$) yields the conclusion that $\|\hat{f}_{\Phi}-f\|_{L^q}\leq C\, n^{-\frac{1}{d+2}}\, \log n$ with high probability. The bounds for $q\geq 1+2/d$ follow from H\"older's inequality between $L^{1+2/d}$ and $L^{\infty}$. The proof for $d=1$ follows the same lines, but it is slightly different. See Section~\ref{Sec:Main_Proof} for the full proof.

The inequality~\eqref{Intro:DummyInt} is sharp, in the sense that the norms in the right-hand side cannot \textit{both} be replaced by weaker norms. In this sense, it is important that our estimator~\eqref{Intro:TV_est1} combines a bound on the frame coefficients (related to the $\besov$-norm) with control on the $\bv$-seminorm. 
Finally, notice that the argument above does not rely on Gaussianity of the process $dW$: it holds whenever the random variables $\int \phi_{\omega}(x)\, dW(x)$ have subgaussian tails.

\begin{example}\label{Intro:Example1}
In order to illustrate the performance of the estimator $\hat{f}_{\Phi}$, consider the situation where $d=2$ and the dictionary $\Phi$ consists of normalized indicator functions of dyadic squares~\citep{Nemirovski},
\begin{equation*}
\Phi=\bigg\{ \frac{1}{\sqrt{|B|}}\, 1_{B}(x)\,\bigg|\, B \textup{ dyadic square }\subseteq [0,1]^2\bigg\},
\end{equation*}
where $|B|$ denotes the Lebesgue measure of the set $B$. 
Now, the estimator $\hat{f}_{\Phi}$ in~\eqref{Intro:TV_est1} becomes
\begin{equation}
 \hat{f}_{\Phi}\in\underset{g\in X_n}{\textup{ argmin }}|g|_{\bv}\ \textup{ s.t.}  \max_{\textup{dyadic } |B|\geq \frac{1}{n}}\frac{1}{\sqrt{|B|}}\bigg|\int_Bg(x)-f(x)\, dx-\frac{\sigma}{\sqrt{n}}\int_BdW(x)\bigg|\leq\gamma_n,\label{Intro:TV_estEx}
\end{equation}
that is, $\Omega_n$ consists of all squares $B\subseteq[0,1]^2$ of area $|B|\geq 1/n$ with vertices at dyadic positions. The main peculiarity of $\hat{f}_{\Phi}$ is the data-fidelity term, which encourages proximity of $\hat{f}_{\Phi}$ to the truth $f$ \textit{simultaneously} at all dyadic squares $B$. This results in an estimator that preserves features of the truth in both the large and the small scales, thus giving a \textit{spatially adaptive} estimator. This is illustrated in Figure~\ref{fig:Intro1} (see~\cite{frick2013statistical} for computational details): the estimator $\hat{f}_{\Phi}$ succeeds to reconstruct the image well at both the large (sky and building) and small scales (stairway). For comparison we also show the classical TV-regularization estimator, a.k.a.~Rudin-Osher-Fatemi (ROF) estimator~\citep{ROF}, defined in~\eqref{Intro:ROF}, which employs a global $L^2$ data-fidelity term. 
The parameter $\lambda_n$ in~\eqref{Intro:ROF} is chosen in an oracle way so as to minimize the distance to the truth, which serves as a benchmark for any data-driven parameter choice. Here we measure the "distance" by the symmetrized Bregman divergence of the $BV$ seminorm (see Section 3 of~\cite{frick2012} for a motivation for this and other distances). The ROF estimator successfully denoises the image in the large scales at the cost of losing details in the small scales. The reason is simple: the use of the $L^2$ norm as a data-fidelity, which only measures the proximity to the data \textit{globally}. As a consequence, the optimal parameter $\lambda_n$ is forced to achieve the best trade-off between regularization and data fidelity \textit{in the whole image}: in particular, in rich enough images there will be regions where one either over-regularizes or under-regularizes, e.g.~in the stairway in Figure~\ref{fig:Intro1}(d).

\begin{figure}

 \begin{center}
 \begin{tikzpicture}

  \node[anchor=south west,inner sep=0] at (0,1.4)
      {\includegraphics[width=.5\linewidth]{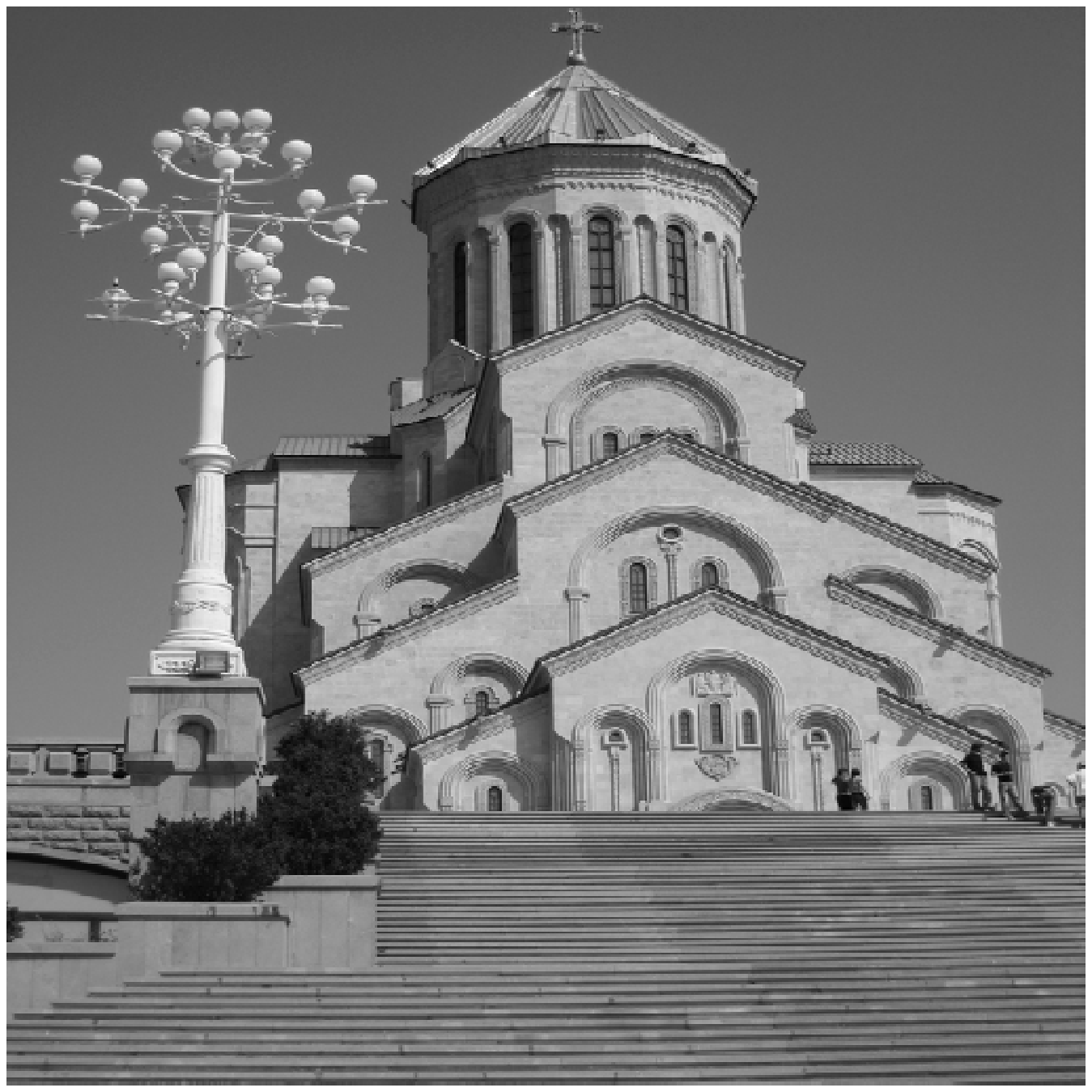}};

 \node[anchor=north] at (0.6,8) {(a)};
  
  \node[anchor=south west,inner sep=0] at (8,4.5) {\includegraphics[width=.5\linewidth,trim={{0.65\linewidth} {0.5\linewidth} 0 {0.28\linewidth}},clip]{./Church_Original.eps}};
      
      \draw[red,thin] (2.8,2.15) rectangle ++(4.12,1.32);  
      \draw[red,thin] (4.5,4.85) rectangle ++(2.4,1.5);   
      
      \node[anchor=south west,inner sep=0] at (8,1) {\includegraphics[width=.5\linewidth,trim={{0.4\linewidth} 0 0 {0.7\linewidth}},clip]{./Church_Original.eps}};

  \node[anchor=south west,inner sep=0] at (0,-5)
      {\includegraphics[width=.5\linewidth]{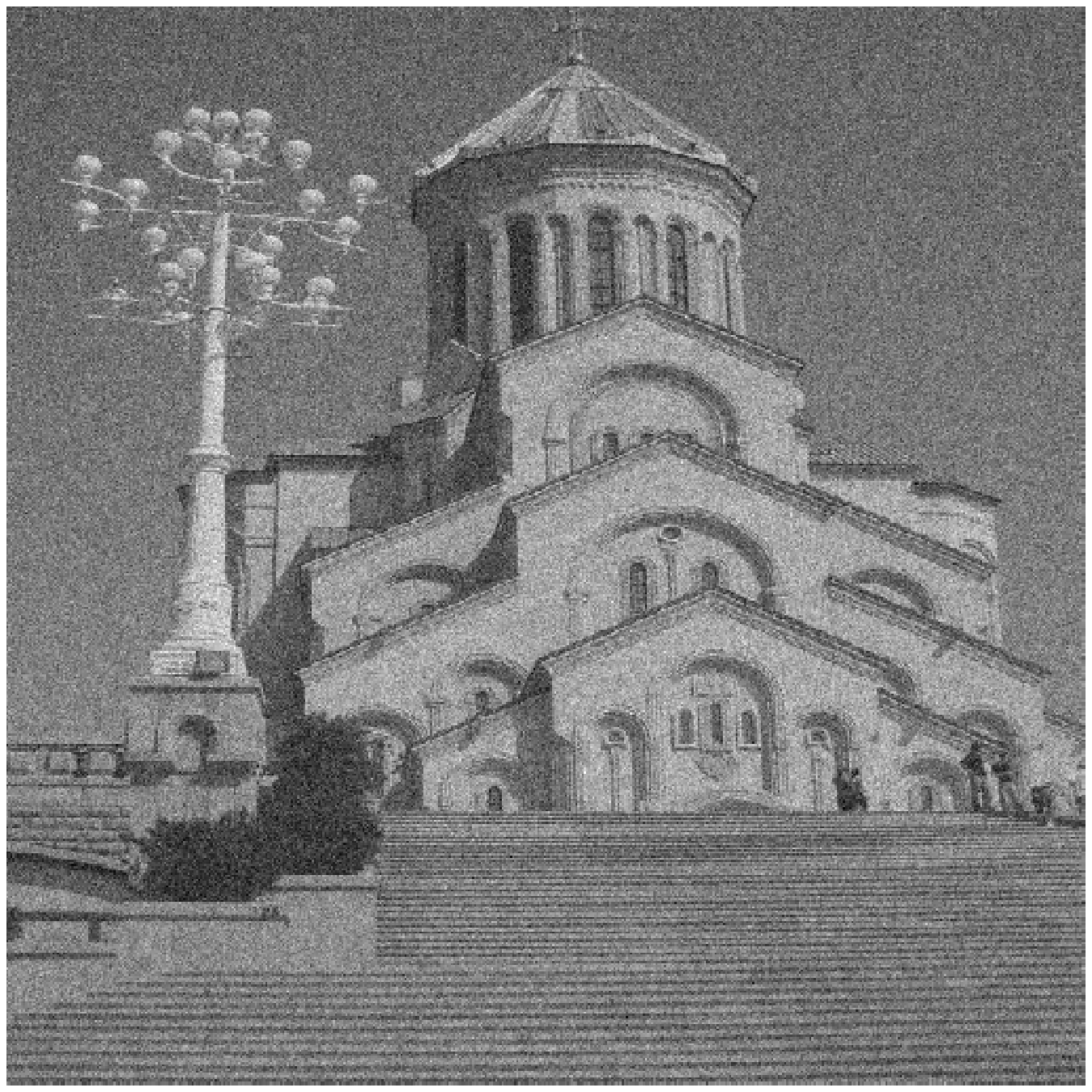}};

 \node[anchor=north] at (0.6,1.6) {(b)};
  
  \node[anchor=south west,inner sep=0] at (8,-1.9) {\includegraphics[width=.5\linewidth,trim={{0.65\linewidth} {0.5\linewidth} 0 {0.28\linewidth}},clip]{./Church_SNR5.eps}};
      
      \draw[red,thin] (2.8,2.15-6.4) rectangle ++(4.12,1.32);  
      \draw[red,thin] (4.5,4.85-6.4) rectangle ++(2.4,1.5);   
      
      \node[anchor=south west,inner sep=0] at (8,-5.4) {\includegraphics[width=.5\linewidth,trim={{0.4\linewidth} 0 0 {0.7\linewidth}},clip]{./Church_SNR5.eps}};

 \node[anchor=north] at (0.5,-5) {(c)};
 
  \node[anchor=south west,inner sep=0] at (1,-1.9-6.4) {\includegraphics[width=.5\linewidth,trim={{0.65\linewidth} {0.5\linewidth} 0 {0.28\linewidth}},clip]{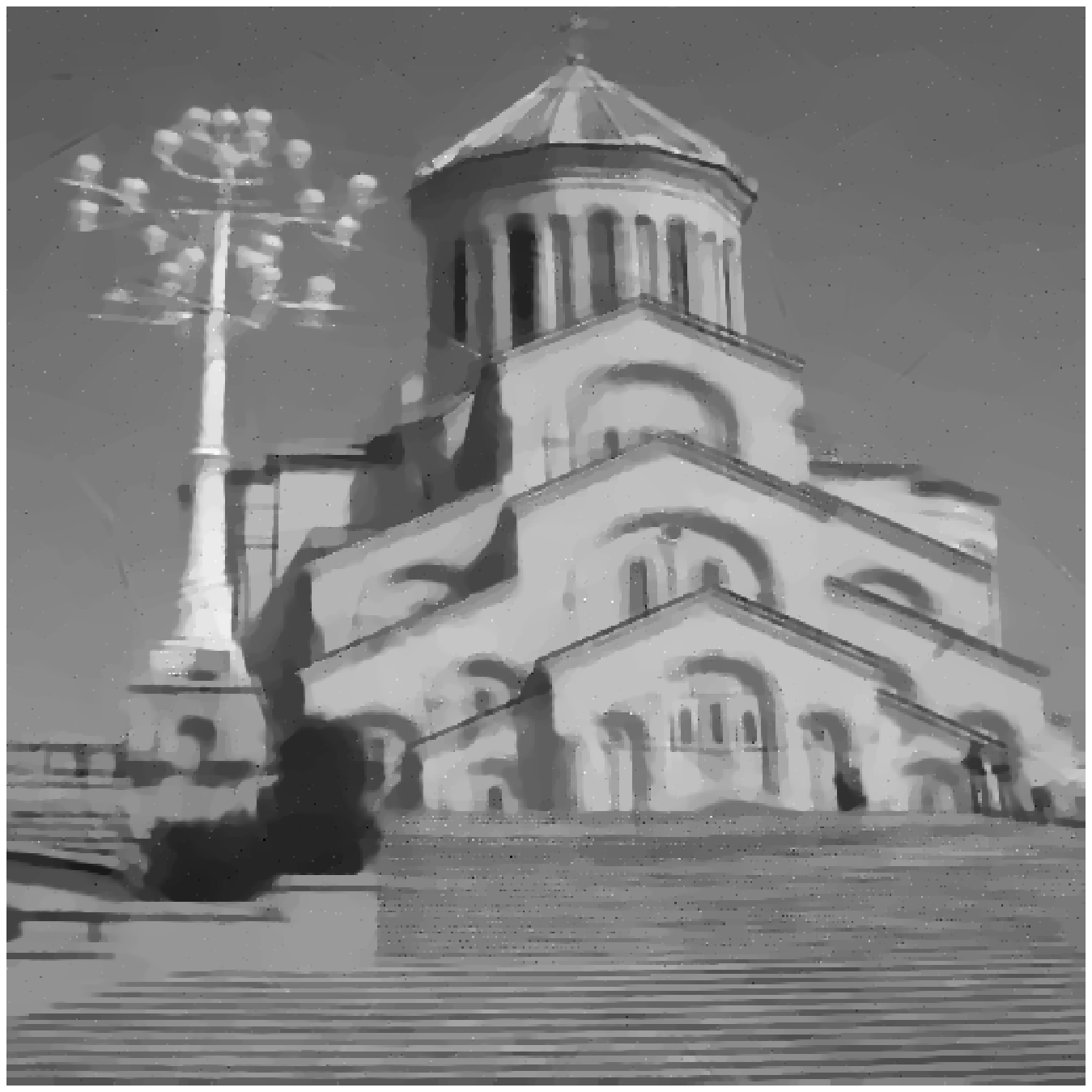}};

      \node[anchor=south west,inner sep=0] at (1,-5.4-6.4) {\includegraphics[width=.5\linewidth,trim={{0.4\linewidth} 0 0 {0.7\linewidth}},clip]{./Church_S5_MRTV_q25.eps}};

 \node[anchor=north] at (7.5,-5) {(d)};
 
  \node[anchor=south west,inner sep=0] at (8,-1.9-6.4) {\includegraphics[width=.5\linewidth,trim={{0.65\linewidth} {0.5\linewidth} 0 {0.28\linewidth}},clip]{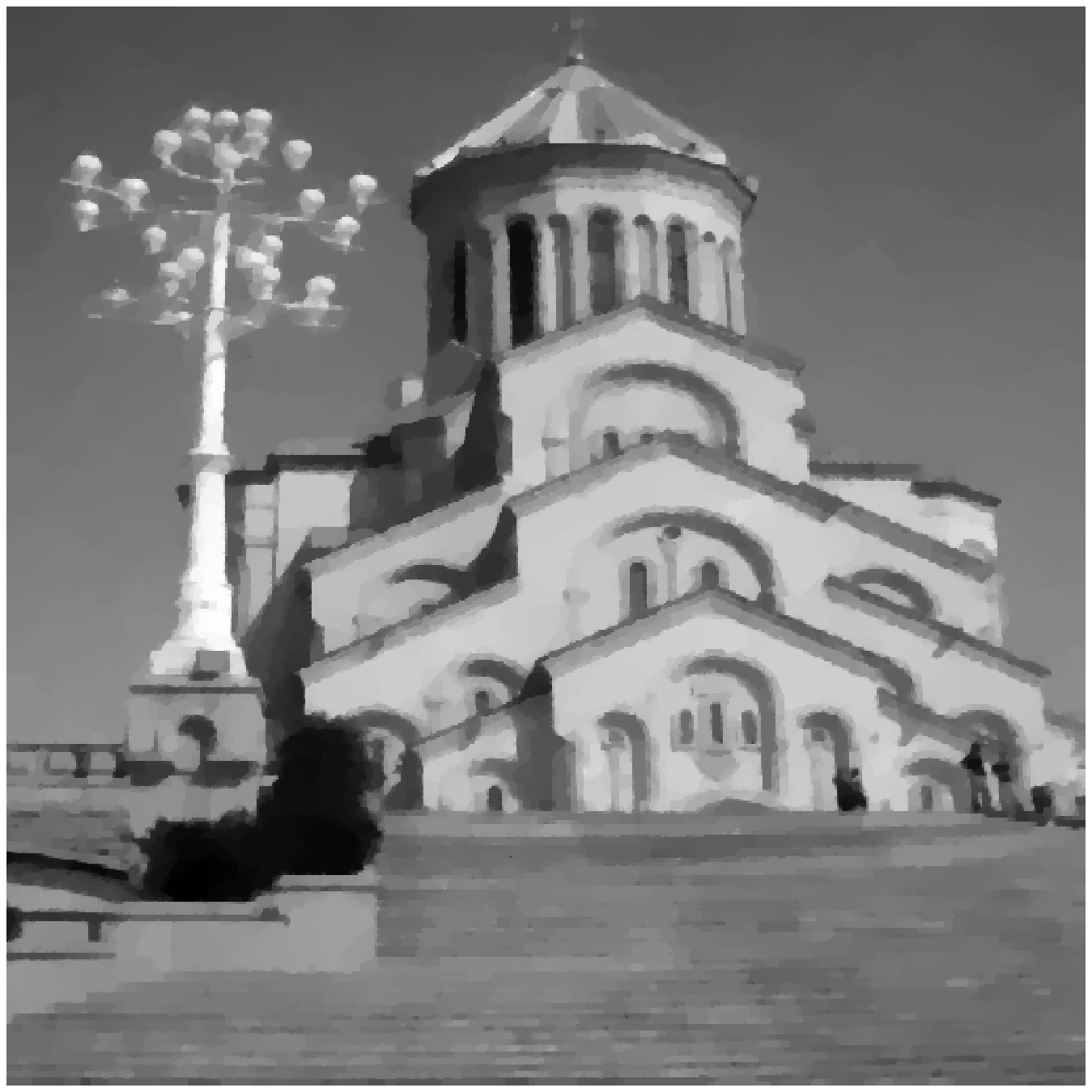}};

      \node[anchor=south west,inner sep=0] at (8,-5.4-6.4) {\includegraphics[width=.5\linewidth,trim={{0.4\linewidth} 0 0 {0.7\linewidth}},clip]{./Church_SN5_L2tv_BregOrac.eps}};
 \end{tikzpicture}
 \end{center}
 \vspace{-0cm}
 \caption{(a) Original image, (b) noisy version with signal-to-noise ratio $\textup{SNR}=5$, (c) zoom in of the multiscale TV estimator~\eqref{Intro:TV_estEx} with $\kappa=1/2$ in~\eqref{Res:UnivGamma}, and (d) zoom in of the estimator $\hat{f}_{\lambda_n}$ from~\eqref{Intro:ROF} with oracle $\lambda_n^*=\textup{argmin}\, \mathbb{E}\big[D_{BV}(\hat{f}_{\lambda_n},f)\big]$, where $D_{BV}(\cdot,\cdot)$ denotes the symmetrized Bregman divergence of the $BV$ seminorm.}\label{fig:Intro1}
\end{figure}

\end{example}

\begin{otherex}
Other estimators that minimize the $BV$ seminorm and fall into our framework~\eqref{Intro:TV_est1}, covered by Theorem~\ref{Res:Main_thm}, result from dictionaries $\Phi$ consisting of a wavelet basis (\cite{donoho1993unconditional},~\cite{hardle2012wavelets}), a curvelet frame~\citep{candes2000curvelets} or a shearlet frame~\citep{labate2005sparse}. 
Such estimators have been proposed in the literature (\cite{CandesGuo},~\cite{frick2012},~\cite{malgouyres2002}) and have been shown to perform very well in simulations, outperforming wavelet and curvelet thresholding, and TV-regularization with global $L^2$ data-fidelity, as illustrated in Figure~\ref{fig:Intro1}.
\end{otherex}

\subsection*{Related work}

This paper is related to a number of results at the cutting edge of statistics, mathematical imaging and applied harmonic analysis. As the literature is vast, we only mention some selective references. 
Starting with the seminal paper~\cite{ROF} that proposed the TV-penalized least squares estimator~\eqref{Intro:ROF} for image 
denoising (the ROF estimator), the subsequently developed theory of TV-based estimators depends greatly on the spatial dimension. 
In dimension $d=1$,~\cite{mammen1997} showed that the ROF-estimator attains the optimal rate of convergence in the discretized nonparametric 
regression model, and~\cite{donoho1998minimax} proved that wavelet thresholding for estimation over $BV$ attains the minimax rates with the exact logarithmic factors. 
We also refer to~\cite{davies2001} and~\cite{dumbgen2009} for a combination of TV-regularization with related multiscale data-fidelity terms in $d=1$, and to~\cite{frick2014} and~\cite{li2017multiscale} for the combination of a multiscale constraint with a jump penalty for segmentation of one-dimensional functions

In higher dimensions, the situation becomes more involved due to the low regularity of functions of bounded variation. There are roughly two approaches to deal with this: either employ a finer data-fidelity term, or discretize the problem. Concerning the first approach, we distinguish three different variants that are related to our work. 
First,~\cite{Meyer} proposed the replacement of the 
$L^2$-norm in the ROF functional by a weaker norm designed to match the smoothness of Gaussian noise. 
Several algorithms and theoretical frameworks using the Besov norm $B^{-1}_{\infty,\infty}$~\citep{garnett2007}, 
the $G$-norm~\citep{haddad2007} 
and the Sobolev norm $H^{-1}$ in $d=2$~\citep{osher2003image} were proposed, but the statistical performance of these 
estimators has not been analyzed. 
A second variant (see~\cite{durand2001artifact},~\cite{malgouyres2001unified} and~\cite{malgouyres2002}) involved 
estimators of the form~\eqref{Intro:TV_est1} with a wavelet basis $\Phi$. Following this approach and the development of 
curvelets (see e.g.~\cite{candes2000curvelets} for an early reference),~\cite{CandesGuo} and~\cite{starck2001} proposed the 
estimator~\eqref{Intro:TV_est1} with $\Phi$ being a curvelet frame and a mixed curvelet and wavelet family, respectively, which showed good numerical behavior. 
The third line of development that leads to the estimator~\eqref{Intro:TV_est1} is based on Nemirovski's work~\cite{nemirovskii1985nonparametric}, who credits S.~V.~Shil'man for the original idea (see also~\cite{Nemirovski}), and on Donoho's work on soft-thresholding~\cite{donoho1993unconditional}. Nemirovski proposed 
a variational estimator for nonparametric regression over H\"older and Sobolev spaces that used a data-fidelity term based on the 
combination of local likelihood ratio (LR) tests: the \textit{multiresolution norm}. 
In statistical inverse problems,~\cite{dong2011automated} proposed an estimator using TV-regularization constrained by the \textit{sum} of local averages of residuals, instead of the maximum we employ in~\eqref{Intro:TV_est1}, which was proposed by~\cite{frick2012}. 
Finally, during revision of this work we became aware of the work by~\cite{fang2019}, who consider estimation of functions of bounded variation in the sense of Hardy-Krause. This class of functions has higher regularity than $BV$, and hence is much smaller: it corresponds roughly to Sobolev $W^{d,1}$ functions, i.e., with $d$ partial derivatives in $L^1$, which explains the faster minimax rate $n^{-1/3}$ in any dimension.

The other approach to TV-regularization in higher dimensions is to discretize the observational model~\eqref{Intro:WNmodel}, thereby reducing the problem of estimating a function $f\in BV$ to that of estimating a vector of function values $(f(x_1),\ldots,f(x_n))\in\mathbb{R}^n$, where $\{x_i\}$ are design points in $[0,1]^d$. In particular, the risk is measured by the \textit{Euclidean norm} of $\mathbb{R}^n$, and not by the continuous $L^2$-norm. TV-regularized least squares in this discrete setting is nowadays fairly well understood. We mention~\cite{dalalyan2017} and~\cite{hutter2016optimal}, who proved convergence rates in any dimension $d$, which were shown to be minimax optimal in that model~\cite{sadhanala2016total}. Its generalization to trend-filtering, where higher order derivatives are assumed to belong to $BV$, is a current research topic~\cite{guntuboyina2017spatial},~\cite{wang2016trend}. However, this discretized model is substantially different from the continuous model that we consider. In fact, the works just mentioned deal with a finite dimensional parameter space of discretized signals and regularize with the $\ell^1$-norm of the discrete gradient, which in the limit of finer discretization converges to the Sobolev $W^{1,1}$ seminorm. Hence, $BV$ functions are indistinguishable from Sobolev $W^{1,1}$ functions in the discretized model for any dimension $d\in\mathbb{N}$. However, the difference between $W^{1,1}$ and $BV$ functions is significant: while the gradients of the former are finite Lebesgue continuous measures, the gradients of the latter can be any finite Radon measure, i.e.~Lebesgue singular measures are allowed.
Consequently, $BV$ functions can have jump singularities, which makes their estimation significantly more challenging than estimating a Sobolev function. 
Therefore, in contrast to the analysis of discrete TV-regularization, the continuous setting is more subtle and genuinely analytical tools are needed, such as the interpolation inequality~\eqref{Intro:DummyInt}. Moreover, a limitation of discretized models is that they typically discretize the functions and the TV functional with respect to the \textit{same} grid. The discretization of the signals is usually determined by the application, while different discretizations of the TV functional can have different effects (see e.g.~\cite{condat2017}). It is hence useful to study the estimation of $BV$ functions in the continuous setting, since it gives insight into the estimation problem, independently of the discretization of signals or functionals.

Regarding the tools and techniques we use, we mention in particular the concept of an interpolation inequality that relates the risk functional, the regularization functional and the data-fidelity term (see~\cite{nemirovskii1985nonparametric} and~\cite{grasmair2015}). While the inequality in those papers is essentially the Gagliardo-Nirenberg inequality for Sobolev norms (see Lecture II in~\cite{nirenberg1959}), we extend and make use of interpolation inequalities for the $BV$ norm, e.g. equation~\eqref{Intro:DummyInt}, see~\cite{cohen2003harmonic} and~\cite{ledoux2003}. 
Finally, as opposed to~\cite{grasmair2015}, we formulate our results in the white noise model. This eases the incorporation of results from harmonic analysis (e.g.~the interpolation inequalities between $BV$ and $\besov$ and the characterization of Besov spaces by local means) into our statistical analysis, as discretization effects (due to sampling) do not occur. See, however, Section~\ref{Sect:Summary} for a discussion of our results in the latter case.

\subsubsection*{Organization of the paper}

In Section~\ref{Sect:Res} we state general assumptions on the family $\Phi$ under which the estimator $\hat{f}_{\Phi}$ is shown to be nearly minimax optimal over the set $BV_L$. We give a complete statement of the Main Theorem. Then we present examples of the estimator~\eqref{Intro:TV_est1} where $\Phi$ is a 
wavelet basis, a multiresolution system, and a curvelet or shearlet frame combined with wavelets, and show their almost minimax optimality for $L^q$-risks, $q\geq 1$. The proof of the main theorem is given in Section~\ref{Sec:Main_Proof}, while several analytical results are relegated to the Supplement. In Section~\ref{Sect:Summary} we briefly discuss possible extensions.

\subsubsection*{Notation}
We denote the Euclidean norm of a vector $v=(v_1,\ldots,v_d)\in\Rd$ by $|v|:=\big(v_1^2+\cdots+v_d^2\big)^{1/2}$. 
For a real number $x$, define $\lfloor x\rfloor:=\textup{max}\big\{m\in\mathbb{Z}\, \big|\, m\leq x\big\}$ 
and $\lceil x\rceil:=\textup{min}\big\{m\in\mathbb{Z}\, \big|\, m>x\big\}$. The cardinality of a finite set $X$ is denoted by 
$\# X$. We say that two norms $\|\cdot\|_{\alpha}$ and $\|\cdot\|_{\beta}$ in a normed space $V$ are equivalent, and write $\|v\|_{\alpha}\asymp \|v\|_{\beta}$, if there are constants $c_1,c_2>0$ such that $c_1\|v\|_{\alpha}\leq \|v\|_{\beta}\leq c_2\|v\|_{\alpha}$ for all $v\in V$. 
Finally, we denote by $C$ a generic positive constant that may change from line to line.

\newpage

\section{Results}\label{Sect:Res}

\subsection{Basic definitions}\label{Res:Basic_Def}

For $k\in\mathbb{N}$, let $C^k$ denote the space of $k$-times continuously differentiable periodic functions on $[0,1)^d$, which we identify with the $d$-torus $\Td$. 
The space of $1$-periodic functions of bounded variation $BV$ consists of functions $g\in L^1$ whose weak distributional gradient $\nabla g=(\partial_{x_1}g,\cdots,\partial_{x_d}g)$ is a periodic, $\mathbb{R}^d$-valued finite Radon measure on $[0,1)^d$~\cite{Evans}. 
The finiteness implies that the bounded variation seminorm of $g$, defined by
\begin{equation*}
 |g|_{BV}:=\sup\bigg\{\int_{\mathbb{T}^d}g(x)\, \mathrm{div}(h(x))\, dx\, \bigg|\, h\in C^1(\mathbb{T}^d;\mathbb{R}^d),\ \|h\|_{L^{\infty}}\leq 1\bigg\},
\end{equation*}
is finite, where $\mathrm{div}(h)$ denotes the divergence of the vector field $h$. $BV$ is a Banach space with the norm 
$\|g\|_{BV}=\|g\|_{L^1}+|g|_{BV}$, see~\cite{Evans}. 
For $S\in\mathbb{N}$, let $\Phi=\big\{\psi_{j,k,e}\, \big|\, (j,k,e)\in\Omega\big\}$ 
be an $S$-regular wavelet basis for $L^2$ whose 
elements are $S$ times continuously differentiable with absolutely integrable $S$-th derivative, indexed by the set
\begin{align}
 \Omega&:=\big\{(j,k,e)\, \big|\, j\geq 0,\ k\in P_j^d, e\in E_j\big\},\ \ \textup{ with }\label{Res:WaveIndex}
 \\
 P_j^d&:=\big\{k=(k_1,\ldots,k_d)\,\big|\, k_i=0,\ldots,2^j-1,\ i=1,\ldots,d\big\},\nonumber
 \\
 E_j&:=\begin{cases}
       \{0,1\}^d & \textup{ if } j=0,
       \\
       \{0,1\}^d\backslash(0,\ldots,0) & \textup{ else.}
      \end{cases}\nonumber
\end{align}
In particular, we consider wavelets of the form
\begin{equation*}
\psi_{j,k,e}(x)=2^{jd/2}\psi_e\big(2^jx-k\big),
\end{equation*}
where $\psi_e(z_1,\cdots,z_d)=\prod_{i=1}^d\psi_{e_i}(z_i)$ is a tensor product of periodized one-dimensional wavelets, and
\begin{equation*}
\psi_{e_i}(\cdot)=\begin{cases}
\psi(\cdot) \textup{ if } e_i=1,
\\
\varphi(\cdot) \textup{ else},
\end{cases}
\end{equation*}
denotes either the mother wavelet or the father wavelet of a one-dimensional wavelet basis of $L^2$. The index $(0,\cdots,0)\in E_0$ refers here to (shifts of) the father wavelet $\psi_{0,k,0}=\varphi(\cdot-k)$. 
See e.g.~Section 4.3.6 in~\cite{gine2015mathematical} for the construction of such a basis. Then for $p,q\in[1,\infty]$ and 
$s\in\mathbb{R}$ with $S>|s|$, the Besov norm of a (generalized) function is defined by 
\begin{equation}
 \|g\|_{B^s_{p,q}}:=\bigg(\sum_{j\in\mathbb{N}_0}2^{jq\big(s+d(\frac{1}{2}-\frac{1}{p})\big)}\bigg(\sum_{k\in P_j^d}\sum_{e\in E_j}|\langle \psi_{j,k,e},g\rangle|^p\bigg)^{q/p}\bigg)^{1/q},\label{Res:BesovNorm}
\end{equation}
with the usual modifications if $p=\infty$ or $q=\infty$. If $s> 0$ and $p\in[1,\infty)$, the Besov space $B^s_{p,q}$ consists of $L^p$ functions with finite Besov norm, while if $s> 0$ and $p=\infty$, then $B^s_{p,q}$ consists of 
continuous functions with finite Besov norm. In these cases, $\langle \cdot,\cdot\rangle$ denotes the standard inner product in 
$L^2$. If $s\leq 0$, $B^s_{p,q}$ consists of periodic distributions $\mathcal{D}^*(\Td)$ 
with finite Besov norm. Here, $\mathcal{D}^*(\Td)$ denotes the space of 
periodic distributions, defined as the topological dual to the space of infinitely differentiable periodic functions 
$C^{\infty}(\Td)$ (see Section 4.1.1 in~\cite{gine2015mathematical}). In that case, $\langle \psi_{j,k,e},g\rangle$ is 
interpreted as the action of $g\in\mathcal{D}^*(\Td)$ on the function $\psi_{j,k,e}$.

Finally, we define the Fourier transform of a function $g\in L^1(\Td)$ by
\begin{equation}
 \mathcal{F}[g](\xi):=\int_{\Td}g(x)\, e^{-2\pi i \xi x}\, dx, \ \ \ \xi\in\Zd.\label{Res:Notation1}
\end{equation}
The Fourier transform of a function $g\in L^1(\Rd)$ is defined as in~\eqref{Res:Notation1} extending the integration over $\Rd$. 
The formal definition of the Fourier transform is as usual extended to functions in $L^2$ and, by duality, to 
distributions $\mathcal{D}^*(\Td)$ (see e.g.~Section 4.1.1 in~\cite{gine2015mathematical}).

\subsection{Main result}\label{Sec:Main}

The main ingredient of the estimator~\eqref{Intro:TV_est1} is the dictionary $\Phi$, on which we impose the 
following assumptions.
\begin{assumption}\label{Res:Gen_Ass}
 $\Phi$ is of the form $\Phi=\{\phi_{\omega}\, \big|\, \omega\in\Omega\}\subset L^2$ for a countable set $\Omega$ and 
functions satisfying 
$\|\phi_{\omega}\|_{L^2}=1$ for all $\omega\in\Omega$. 
For each $n\in\mathbb{N}$, consider a subset $\Omega_n\subset\Omega$ of polynomial growth, meaning that $c\, n^{\Gamma}\leq \#\Omega_n\leq Q(n)$ for all $n$ for a polynomial $Q$ and constants $c,\Gamma>0$. The sets $\Omega_n$ are assumed to satisfy the inequality~\eqref{Intro:BesovCompat} for any $g\in L^{\infty}$.
\end{assumption}

\begin{examples}
 \quad
 \begin{itemize}
  \item[a)] The simplest example of a system $\Phi$ satisfying Assumption~\ref{Res:Gen_Ass} is a sufficiently smooth wavelet basis. Indeed, 
  the assumption follows from the characterization of Besov spaces in terms of wavelets (see Proposition~\ref{Res:Veri_Wavelet} below).
  \item[b)] Another family $\Phi$ satisfying Assumption~\ref{Res:Gen_Ass} is given by translations and rescalings of 
  (the smooth approximation to) the indicator function of a cube. In Section~\ref{Ex:MultiEst} we verify the assumption for such 
  a system, that has been used previously as a dictionary for function estimation (see~\cite{grasmair2015}).
  \item[c)] In Section~\ref{Ex:sheCurv} we show that frames containing a smooth wavelet basis and a curvelet or a shearlet frame (which play a prominent role in imaging) satisfy 
  Assumption~\ref{Res:Gen_Ass}.
 \end{itemize}
\end{examples}

\begin{definition}
Assume the model~\eqref{Intro:WNmodel}, and let $Y_{\omega}$ be as in~\eqref{Intro:Ycoeff} the projections of the white noise model onto a dictionary $\Phi$ satisfying Assumption~\ref{Res:Gen_Ass}. We denote the estimator in~\eqref{Intro:TV_est1} as \textit{frame-constrained TV-estimator} with respect to the dictionary $\Phi$, where we minimize over the set
\begin{equation}
 X_n:=\big\{g\in BV\cap L^{\infty}\, \big|\, \|g\|_{L^{\infty}}\leq \log n\big\}.\label{Res:Aux_Set}
\end{equation}
We use the convention in~\eqref{Intro:TV_est1} that, whenever the \textup{argmin} is taken over the empty set, $\hat{f}_{\Phi}$ is the constant zero function.
\end{definition}

In the following we assume that $n\geq 2$, so that we do not have to worry about the case $\log 1=0$. The reason for the additional constraint $\|g\|_{L^{\infty}}\leq \log n$ is technical: We will need upper bounds on the supremum norm of $\hat{f}_{\Phi}$. As it turns out, the upper bound $\log n$ will not affect the minimax polynomial rate of convergence of the estimator (but it yields additional logarithmic factors in the risk). Alternatively, if we knew an upper bound $L$ for the supremum norm of $f$, we could choose $X_n=\{g\in BV\cap L^{\infty}\, |\, \|g\|_{L^{\infty}}\leq L\}$. In that case, the risk bounds in Theorem~\ref{Res:Main_thm}  would improve in some logarithmic factors (see Remark~\ref{Res:RemPolyLog}).

\begin{theorem}\label{Res:Main_thm}
Let $d\in\mathbb{N}$, and assume the model~\eqref{Intro:WNmodel} with $f\in BV_L$ for some $L>0$. Let further $q\in \big[1,\infty\big)$.
\begin{itemize}
 \item[a)] Let $\gamma_n$ be as in~\eqref{Res:UnivGamma} with $\kappa>1$, and let $\Phi$ be a family of functions satisfying Assumption~\ref{Res:Gen_Ass}. Then for any $n\in\mathbb{N}$ with 
 $n\geq e^L$, the estimator $\hat{f}_{\Phi}$ in~\eqref{Intro:TV_est1} with parameter $\gamma_n$ satisfies
 \begin{equation}
    \sup_{f\in BV_L}\|\hat{f}_{\Phi}-f\|_{L^q}\leq C\, n^{-\min\{\frac{1}{d+2},\frac{1}{dq}\}}\, (\log n)^{3-\min\{d,2\}}\label{Res:Gen_Conv}
 \end{equation}
with probability at least $1-\big(\#\Omega_n\big)^{1-\kappa^2}$.
\item[b)] Under the assumptions of part a),  if $\kappa^2>1+\frac{1}{(d+2)\, \Gamma}$ with $\Gamma$ as in Assumption~\ref{Res:Gen_Ass}, then
 \begin{equation}
    \sup_{f\in BV_L}\mathbb{E}\big[\|\hat{f}_{\Phi}-f\|_{L^q}\big]\leq C\, n^{-\min\{\frac{1}{d+2},\frac{1}{dq}\}}\, (\log n)^{3-\min\{d,2\}}\label{Res:Gen_Conv_Exp}
 \end{equation}
 holds for $n$ large enough and a constant $C>0$ independent of $n$.
\end{itemize}
\end{theorem}

\begin{remark}\label{Res:RemNoise}
\quad
 \begin{itemize}
 \item[a)] Notice that part a) of the theorem implies that~\eqref{Res:Gen_Conv} holds asymptotically almost surely if $\kappa^2>2$.
  \item[b)] By the assumption that $\|\phi_{\omega}\|_{L^2}=1$ $\forall\omega\in\Omega$, we have the tail bound
  \begin{equation*}
  \mathbb{P}\bigg(\max_{\omega\in\Omega_n}\bigg|\int_{\mathbb{T}^d}\phi_{\omega}(x)\, dW(x)\bigg|\geq t\bigg)\leq \#\Omega_n\, e^{-t^2/2},
\end{equation*}
  for any $n\in\mathbb{N}$ and $t\geq 0$, where $dW$ denotes the white noise process in $L^2(\Td)$. This bound follows from Chernoff's inequality and the union bound, and it will play an important role for bounding the stochastic estimation error of the estimator $\hat{f}_{\Phi}$.
 \end{itemize}
\end{remark}

\begin{remark}\label{Res:RemPolyLog}
The logarithmic factors in~\eqref{Res:Gen_Conv} and~\eqref{Res:Gen_Conv_Exp} are equal to $(\log n)^2$ for $d=1$ and to $\log n$ for $d\geq 2$. They arise in part from the bound $\|\hat{f}_{\Phi}\|_{L^{\infty}}\leq \log n$ (that we get from minimizing over $X_n$ in~\eqref{Res:Aux_Set}), while part of them arise from the estimation procedure itself. Indeed, if we additionally constrain the estimator to $\|\hat{f}_{\Phi}\|_{L^{\infty}}\leq C$, the factors can be improved to $(\log n)^{1+\min\{\frac{1}{d+2},\frac{1}{dq}\}}$ 
and $(\log n)^{\min\{\frac{1}{d+2},\frac{1}{dq}\}}$ for $d=1$ and $d\geq 2$, respectively. See Proposition~\ref{InterpInequality} in Section~\ref{Sec:Main_Proof} for an explanation of 
the different factors in $d=1$ and $d\geq 2$.
\end{remark}

\begin{remark}\label{Res:SharpJackson}
Recall that our parameter set $BV_L$ involves a bound on the supremum norm. 
This bound can be relaxed to a bound on the Besov $B^0_{\infty,\infty}$ norm without 
changing the convergence rate $n^{-\min\{\frac{1}{d+2},\frac{1}{dq}\}}$ for $\hat{f}_{\Phi}$. 
Indeed, assume for simplicity that $\Phi$ is an orthonormal wavelet basis of $L^2$, and for $n\in\mathbb{N}$ let $\Omega_n$ 
index the wavelet coefficients up to level $J=\big\lfloor\frac{1}{d}\frac{\log n}{\log 2}\big\rfloor$. In the proof of 
Theorem~\ref{Res:Main_thm} we need a relaxed form of Assumption~\ref{Res:Gen_Ass}, namely an inequality of the form
\begin{equation}
 \max_{(j,k,e)\in\Omega}|\langle\psi_{j,k,e},g\rangle|\leq\max_{(j,k,e)\in\Omega_n}|\langle\psi_{j,k,e},g\rangle|+C2^{-Jd/2}\ \ \forall J\in\mathbb{N}\label{Res:Bernstein1}
\end{equation}
for sufficiently smooth $g$. But this inequality for all $J\in\mathbb{N}$ is equivalent to 
$\|g\|_{B^0_{\infty,\infty}(\Td)}\leq C$ (see Berstein-type inequalities for Besov spaces, e.g.~in Section 3.4 
in~\cite{cohen2003numerical}). Consequently, Theorem~\ref{Res:Main_thm} can be extended to show that the estimator 
$\hat{f}_{\Phi}$ with an orthonormal wavelet basis $\Phi$ attains 
the optimal polynomial rates of convergence uniformly over the enlarged parameter space 
$\widetilde{BV}_L:=\big\{g\in BV\, \big|\, |g|_{BV}\leq L,\ \ \|g\|_{B^0_{\infty,\infty}}\leq L\big\}$. 

One could ask whether the requirement $\|g\|_{B^0_{\infty,\infty}}\leq L$ can be relaxed further. This is not the case if $d\geq 2$. 
Indeed, since the embedding $B^1_{1,\infty}\subset B^{0}_{\infty,\infty}$ holds for $d=1$ only (see~\eqref{Res:BesovNorm}), and since we 
have $BV\subset B^1_{1,\infty}$, we see that a typical function of bounded variation does not belong to 
$B^0_{\infty,\infty}$ if $d\geq 2$. Hence, the Jackson-type inequality in~\eqref{Res:Bernstein1} cannot hold for 
general functions of bounded variation in $d\geq 2$. 
This explains why our parameter space is the intersection of a $BV$-ball with an $L^{\infty}$-ball (or a $B^0_{\infty,\infty}$-ball). Finally, we remark that most works in function estimation deal 
with H\"older or Sobolev functions with $k>d/p$, so the assumption $f\in L^{\infty}$ is implicit. Alternatively, we refer to 
Section 3 in~\cite{lepski1997optimal} and to~\cite{DelyonJuditski} for examples 
of estimation over Besov bodies $B^s_{p,q}$ where uniform boundedness has to be assumed explicitly if $s< d/p$.
\end{remark}

\begin{remark}\label{Res:RemPeriodic}
In this work we deal with the estimation of periodic functions, i.e.~defined on the $d$-torus $\Td$. The reason for that is purely technical: our analysis makes use of Banach spaces of functions, whose definition is simpler for functions defined over $\Td$ (a manifold without boundary) than over the hypercube $[0,1]^d$ (which has a boundary). We remark that our work could be extended to function spaces over $[0,1]^d$ by the use of boundary corrected wavelet bases (see Section 4.3.5 in~\cite{gine2015mathematical}), and adapting the definitions of Besov and $BV$ spaces and their corresponding norms. 
\end{remark}

We can now state the main result of this paper, which is a direct consequence of Theorem~\ref{Res:Main_thm}.

\begin{theorem}\label{MinimaxThm}
Under the assumptions of Theorem~\ref{Res:Main_thm}, the estimator $\hat{f}_{\Phi}$ is minimax optimal up to logarithmic factors over the parameter set $BV_L$ defined in~\eqref{Intro:Param1} with respect to the $L^q$-risk for $q\in\big[1,\infty\big)$ in any dimension $d\in\mathbb{N}$, i.e.,
 \begin{equation*}
    \inf_{\hat{f}}\sup_{f\in BV_L}\mathbb{E}\big[\|\hat{f}-f\|_{L^q}\big]\geq C\, n^{-\min\{\frac{1}{d+2},\frac{1}{dq}\}}\label{Res:Lower}
 \end{equation*}
 for any $q\in[1,\infty)$, where the infimum runs over all measurable functions from the sample space of $dY$ in~\eqref{Intro:WNmodel} to the reals.
\end{theorem}

The proof of Theorem~\ref{MinimaxThm} is given in Section~\ref{SM:minimax}. It consists of proving a lower bound for the minimax risk over $BV_L$, which we show agrees with the upper bound proven in Theorem~\ref{Res:Main_thm}.

\subsection{Examples}\label{Sect:Examps}

\subsubsection{Wavelet-based estimator}\label{Ex:Wavelet}
For $S\in\mathbb{N}$, let $\Phi=\big\{\psi_{j,k,e}\, \big|\, (j,k,e)\in\Omega\big\}$ be an $S$-regular wavelet basis of 
$L^2(\Td)$ as described in Section~\ref{Res:Basic_Def}. For $n\in\mathbb{N}$, $n\geq 2^d$, define the subset
\begin{equation}
 \Omega_n:=\big\{(j,k,e)\in\Omega\, \big|\, j=0,\ldots,J-1\big\},\label{Res:Wave_Sets}
\end{equation}
with $J=\big\lfloor\frac{1}{d}\frac{\log n}{\log 2}\big\rfloor$. Note that $2^{-d}\,n\leq \#\Omega_n=2^{Jd}\leq n$ for any 
$n\geq 2^d$.

\begin{proposition}\label{Res:Veri_Wavelet}
An $S$-regular wavelet basis of $L^2$ as in Section~\ref{Res:Basic_Def} with $S>\max\{1,d/2\}$ satisfies 
Assumption~\ref{Res:Gen_Ass} with the sets $\Omega_n$ in~\eqref{Res:Wave_Sets}, a linear polynomial $Q(x)=x$ and parameter $\Gamma=1$.
\end{proposition}
For the proof, see Section~\ref{SM:VerificationWavelet} in the Supplement. A direct consequence of this proposition and of Theorem~\ref{Res:Main_thm} is that the frame-constrained TV-estimator with the wavelet basis above is nearly minimax optimal for estimating functions in $BV_L$.

\begin{remark}\label{Ex:WaveThresRem}
In dimension $d=1$,~\cite{donoho1998minimax} proved that thresholding of the empirical wavelet coefficients of the observations gives an estimator that attains the minimax optimal convergence rate over $BV$. In contrast, our estimator combines a constraint on the wavelet coefficients with a control on the $BV$-seminorm: this second aspect is crucial in higher dimensions. 
Indeed, in the proof of Theorem~\ref{Res:Main_thm} we bound the risk by the $\besov$-norm of the residuals, which is the maximum of their wavelet coefficients, and the $\bv$-norm of the residuals. The optimality of the estimator~\eqref{Intro:TV_est1} depends crucially on the bound $\|\hat{f}_{\Phi}-f\|_{BV}\lesssim \log n$, which essentially amounts to a bound on the high frequencies of the residuals. But that is precisely the difficulty with wavelet thresholding of $BV$ functions in higher dimensions. To the best of our knowledge, wavelet thresholding has been shown to perform optimally over Besov spaces $B^s_{p,t}$ for $s>d(1/p-1/2)$ only (see e.g.~\cite{DelyonJuditski}). This condition guaranties that the wavelet coefficients of the truth $f$ decay fast enough, which itself allows one to control the high frequencies of the residuals. But that assumption is not satisfies for $BV$ in dimension $d\geq 2$, since we have $B^1_{1,1}\subset BV$, which satisfies $1>d/2$ for $d=1$ only.  
\end{remark}

\subsubsection{$m$-adic multiscale systems}\label{Ex:MultiEst}

We construct the multiscale TV-estimator by choosing $\Phi$ to be a family of smooth functions supported in cubes of different sizes at different locations. Assumption~\ref{ass:multiSystem} makes this precise. For notational simplicity, we sometimes index the set functions in $\Phi$ by the cube $B\subset[0,1)^d$ in which they are supported, and the set of all cubes considered is denoted by $\Omega$.

\begin{assumption}\label{ass:multiSystem}
The system of functions $\Phi=\big\{\phi_{B}\, \big|\, B\in\Omega\big\}$ satisfies the following conditions:
\begin{itemize}
 \item[a)] for fixed $m\in \mathbb{N}$, $m\geq 2$, the set $\Omega$ consists of the intersections with $[0,1)^d$ of all $m$-adic cubes at 
 $m$-adic positions contained in $[0,2)^d$. For each $n\in\mathbb{N}$ with $n\geq m^d$, define 
 $J=\big\lceil\frac{1}{d}\frac{\log n}{\log m}\big\rceil$, $R=J\max\{1,\frac{d}{2}\}$ and
 \begin{align*}
  \mathcal{D}_{R}&:=\big\{\overline{k}=\big(k_1m^{-R},\cdots,k_dm^{-R}\big)\, \big|\, k_i=0,\ldots,m^{R}-1, \ i=1,\ldots,d\big\},
 \\
 \\
 \Omega_n&:=\bigg\{\big(\overline{k}+[0,m^{-j})^d\big)\cap[0,1)^d\, \bigg|\, j=0,\ldots,J-1,\, \overline{k}\in\mathcal{D}_{R} \bigg\};
 \end{align*}
 \item[b)] there is a function $K\in C^{\infty}(\Rd)$ with supp $K\subseteq[0,1)^d$, $|\mathcal{F}[K](\xi)|>0$ in $|\xi|< 2$ and 
 $\|K\|_{L^2(\Rd)}=1$, $\|K\|_{L^{\infty}(\Rd)}\leq 2$ such that all functions $\phi_B\in\Phi$ are given by translation, dilation 
 and rescaling of $K$. More precisely, for each cube $B\in\Omega$ of the form $B=\overline{k}_B+\big[0,|B|^{1/d}\big)^d$, 
 the function $\phi_B\in\Phi$ is given by
\begin{equation*}
 \phi_B(z)=|B|^{-1/2}\, K\big(|B|^{-1/d}(z-\overline{k}_B)\big).
\end{equation*}
\end{itemize}
\end{assumption}

\begin{remark}\label{remAss1}
\quad
\begin{itemize}
\item[a)] An example of a function $K$ satisfying the above assumptions is the ($L^2$-normalized) convolution of the indicator 
function of the cube $\big[\frac{1}{4},\frac{3}{4}\big]^d$ with the standard mollifier. More generally, the Fourier transform of 
the indicator function of the cube $[a,b]\subset[0,1]^d$ satisfies $|\mathcal{F}[1_{[a,b]}](\xi)|>0$ for $|\xi\cdot(b-a)|<1$. 
Taking $K$ to be a smooth approximation to an indicator function, the estimator~\eqref{Intro:TV_est1} is reminiscent of that 
proposed by~\cite{frick2012}.
 \item[b)] For given $m\geq 2$ and $n\in\mathbb{N}$ with $n\geq m^d$, $\#\Omega_n=J\, m^{dR}=J\, m^{d\, J\, \max\{1,d/2\}}$, whence
 \begin{equation*}
 n^{\max\{1,d/2\}}\leq \#\Omega_n\leq n^{\max\{1,d/2\}}\log n.
\end{equation*}
\end{itemize}
\end{remark}

\begin{proposition}\label{Res:veri_multi}
 Let $\Phi=\big\{\phi_B\, \big|\, B\in\Omega\big\}$ satisfy Assumption~\ref{ass:multiSystem}. Then it 
 satisfies Assumption~\ref{Res:Gen_Ass} with polynomial $Q(x)=x^{\max\{1,d/2\}+1}$ and 
 $\Gamma=\max\{1,d/2\}$.
\end{proposition}
 
See Section~\ref{App:Mu} of the Supplement for the proof of Proposition~\ref{Res:veri_multi}. We remark that part of the proof of Proposition~\ref{Res:veri_multi} is based on a characterizations of Besov spaces via local means~\citep{triebel1988characterizations}. Again this proposition together with Theorem~\ref{Res:Main_thm} proves near minimax optimality for the multiscale TV-estimator.

\subsubsection{Shearlet and curvelet estimators}\label{Ex:sheCurv}

Another relevant example of the estimator in~\eqref{Intro:TV_est1} in $d\geq 2$ corresponds to the case when $\Phi$ contains a 
frame of shearlets or curvelets. While classical curvelets are defined for $d=2$ (see e.g.~\cite{candes2000curvelets}), there are several extensions to higher dimensions. In order to simplify and unify the analysis, in this paper we will work with the construction of shearlets in Section 3 of~\cite{Labate2013}, and the curvelet frame from Section 7 of~\cite{borup2007}. The reason for working with these constructions is that they are defined in all dimensions by a partition of frequency space, thus simplifying the notation. We nevertheless remark that the analysis presented here can be easily adapted to other curvelet and shearlet constructions.

Let $\{\overline{\varphi}_{j,\tilde{\theta}}\, \big|\, (j,\tilde{\theta})\in\Xi\}$ denote either the tight shearlet frame or the tight curvelet frame mentioned above. 
Then $\{\varphi_{j,\tilde{\theta}}\, \big|\, (j,\tilde{\theta})\in\Theta\}$ consists of the normalized periodizations of the elements 
$\overline{\varphi}_{j,\tilde{\theta}}$ that have a nonzero overlap with the indicator function of the unit cube, 
i.e., $\int_{[0,1]^d}\overline{\varphi}_{j,\tilde{\theta}}(z)\, dz\neq 0$. 
For simplicity of the notation, we index the elements by 
$(j,\tilde{\theta})\in\Theta\subset\mathbb{N}_0\times\tilde{\Theta}$, 
where $j\geq 0$ plays the role of a scale index, and $\tilde{\theta}$ indexes the position and orientation of the frame 
elements (see the references above for the precise construction in each case). 
In the rest of this section we will consider frames of $L^2(\Td)$ that contain the set 
$\{\varphi_{j,\tilde{\theta}}\, \big|\, (j,\tilde{\theta})\in\Theta\}$.

\begin{assumption}\label{Res:Ass_Curv}
Let $\big\{\psi_{j,k,e}\, \big|\, (j,k,e)\in\Theta^{W}\big\}$ denote an $S$-regular wavelet basis of $L^2(\Td)$ with 
$S>\max\{1,d/2\}$, and let $\big\{\varphi_{j,\tilde{\theta}}\, \big|\, (j,\tilde{\theta})\in\Theta\big\}$ denote 
the set of functions constructed above. Then define 
$\Phi:=\big\{\psi_{j,k,e}\, \big|\, (j,k,e)\in\Theta^{W}\big\}\cup\big\{\varphi_{j,\tilde{\theta}}\, \big|\, (j,\tilde{\theta})\in\Theta\big\}$. 
Further, for $n\in\mathbb{N}$ define $J=\big\lfloor\frac{1}{d}\frac{\log n}{\log 2}\big\rfloor$ 
and let $\Phi_n:=\big\{\psi_{j,k,e}\, \big|\, (j,k,e)\in\Theta^W_n\big\}\cup\big\{\varphi_{j,\tilde{\theta}}\, \big|\, (j,\tilde{\theta})\in\Theta_n\big\}$, where
\begin{align*}
&\Theta_n^{W}:=\big\{(j,k,e)\in\Theta^W\, \big|\, j=0,\ldots,J-1\big\},
\\
&\Theta_n:=\big\{(j,\tilde{\theta})\in\Theta\, \big|\, j=0,\ldots,\widetilde{J}-1\big\},
\end{align*}
where $\widetilde{J}\in\mathbb{N}$ is the largest possible natural number such that $2^{d(J-1)}\leq\#\Theta_n\leq 2^{dJ}$. 
For consistency with the notation in the previous sections, we define the joint index set $\Omega_n:=\Theta_n^W\cup \Theta_n$.
\end{assumption}

\begin{remark}\label{Curv:NeedMixed}
\quad
\begin{itemize}
 \item[a)] The assumption that $\Phi$ contains a wavelet basis as well as a directional frame is crucial. Indeed, the wavelet 
 basis allows us to 
 upper-bound the Besov norm $\besov$ by the maximum over the frame coefficients with respect to $\Phi$, which we need in order to establish Assumption~\ref{Res:Gen_Ass}. Alternatively, if $\Phi$ consisted on a curvelet frame only, the embeddings in 
 Lemma 9 in~\cite{borup2007} together with classical embeddings of Besov spaces (see Remark 4 of Section 3.5.4 in~\cite{SchmeisserTriebel}) would give the bound
\begin{equation*}
 \|g\|_{B^{-d/2}_{\infty,\infty}(\mathbb{R}^d)}\leq C\max_{(j,\tilde{\theta})\in\Theta}2^{j\delta}|\langle \varphi_{j,\tilde{\theta}},g\rangle|
\end{equation*}
for smooth enough functions $g$, and a $\delta>0$ that depends on the dimension. 
Accordingly, the third step in the sketch of the proof of Theorem~\ref{Res:Main_thm} would deteriorate to
\begin{equation*}
 \|\hat{f}_{\Phi}-f\|_{\Besov}\leq C\frac{n^{\delta'}}{\sqrt{n}}\textup{Polylog}_{d,\delta'}(n)
\end{equation*}
for some $\delta'>0$, and a polylogarithmic factor that diverges as $\delta'\rightarrow 0$. This results in a polynomially suboptimal 
rate of convergence. We remark that this limitation 
arises from the suboptimal embeddings between Besov spaces and decomposition space associated with the curvelet frame. The 
situation for the shearlet frame is analogous, as its associated decomposition space equals that of the curvelet frame 
(see Proposition 4.4 in~\cite{Labate2013}).
\item[b)] We make the assumption that $\#\Theta_n\leq 2^{dJ}$ for any $n\in\mathbb{N}$ and $J=\big\lfloor\frac{1}{d}\frac{\log n}{\log 2}\big\rfloor$ 
in order to simplify subsequent computations. The assumption is justified, since the cardinality of $\Theta_n$ behaves indeed 
like $O(2^{dJ})$. In fact, the number of curvelet (or shearlet) frame elements at scale $2^{-j}$ that have a nonzero overlap with the unit cube behaves 
as $2^{dj}$, since there are $O(2^{j+\frac{d-1}{2}j})$ positions and $O(2^{\frac{d-1}{2}j})$ orientations. We refer to Section 8.2 
in~\cite{candes2004new} and~\cite{borup2007} for the details. The claim for the shearlet frame follows from that of 
the curvelet frame by the comparison in Section 4.4 in~\cite{Labate2013}.
\end{itemize}
\end{remark}

The constructions of tight curvelet frames in~\cite{borup2007} and of shearlet frames in~\cite{Labate2013} yield smooth frame 
elements that are exponentially decaying in space. 
We use this to show that the family $\Phi$ satisfies Assumption~\ref{Res:Gen_Ass}.

\begin{proposition}\label{Res:Veri_Curv}
Let $\Phi$ satisfy Assumption~\ref{Res:Ass_Curv} with either the shearlet or the curvelet frame. Then it satisfies 
Assumption~\ref{Res:Gen_Ass} with $Q(x)=2x$ and $\Gamma=1$.
\end{proposition}

The proof of Proposition~\ref{Res:Veri_Curv} is given in Section~\ref{VerificationCurve} of the Supplement. As a consequence, we conclude from Theorem~\ref{Res:Main_thm} that the curvelet TV-estimator is nearly minimax optimal for estimating $BV_L$ functions.

\subsection*{}

We close this section presenting some dictionaries $\Phi$ that do not satisfy Assumption~\ref{Res:Gen_Ass}, where hence Theorem~\ref{Res:Main_thm} does not apply.
\begin{itemize}
\item[a)] Wavelet systems of low smoothness do not satisfy Assumption~\ref{Res:Gen_Ass}. Our result relies crucially on the fact that the Besov spaces $\besov$ and $B^{1}_{1,1}$ can be characterized by the size of wavelet coefficients. For that, wavelet bases with $S-1$ vanishing moments and smoothness $S$ are needed with $S>\max\{1,d/2\}$ (see Section 4.3 in~\cite{gine2015mathematical}). 
\item[b)] For the multiscale TV-estimator in Section~\ref{Ex:MultiEst} we considered a dictionary $\Phi$ consisting on \textit{smoothed} indicator functions of cubes in $[0,1]^d$. The smoothing part is essential, since we need enough regularity in order to bound the Besov $\besov$-norm in terms of this dictionary, which is done by the characterization of Besov spaces by local means (see Section~\ref{App:Mu} of the Supplement).
\item[c)] As argued in part a) of Remark~\ref{Curv:NeedMixed}, a dictionary consisting solely of a curvelet frame or a shearlet frame does not suffice, since the decomposition spaces they generate (in the sense of~\cite{borup2007}) do not match Besov spaces exactly, whence Assumption~\ref{Res:Gen_Ass} does not hold.
\end{itemize}

\section{Proof of the main theorems}\label{Sec:Main_Proof}

\subsection{Proof of part a) of Theorem~\ref{Res:Main_thm}}\label{App:PfThma}

We show the following easy fact as a preparation for the proof of part a) of Theorem~\ref{Res:Main_thm}. 

\begin{proposition}\label{App:aux_Main_thm}
 Let $\Phi$ satisfy Assumption~\ref{Res:Gen_Ass} and, for $n\in\mathbb{N}$, let $\hat{f}_{\Phi}$ be the estimator defined 
 in~\eqref{Intro:TV_est1} with $\gamma_n$ given by~\eqref{Res:UnivGamma}. 
 Then conditionally on the event $A_n$ in~\eqref{Sketch:GoodEvent} we have
 \begin{align*}
  &(i)\ \ \|\hat{f}_{\Phi}-f\|_{\Besov}\leq C\, \gamma_n+C\frac{\|f\|_{L^{\infty}(\Td)}+\log n}{\sqrt{n}},
  \\
  &(ii)\ \ \|\hat{f}_{\Phi}-f\|_{\BV}\leq \|f\|_{L^{\infty}(\Td)}+2|f|_{\BV}+\log n,
 \end{align*}
for any $f\in\BV\cap L^{\infty}(\Td)$, and a constant $C>0$ independent of $n$, $f$ and $\hat{f}_{\Phi}$.
\end{proposition}
\begin{proof}
 For part $(i)$, apply Assumption~\ref{Res:Gen_Ass} to $g=\hat{f}_{\Phi}-f$, which yields
 \begin{equation*}
     \|\hat{f}_{\Phi}-f\|_{\Besov}\leq C\max_{\omega\in\Omega_n}\big|\langle \phi_{\omega},\hat{f}_{\Phi}-f\rangle\big|+C\frac{\|\hat{f}_{\Phi}-f\|_{L^{\infty}(\Td)}}{\sqrt{n}}.
 \end{equation*}
The numerator in the second term can be bounded by $\|f\|_{L^{\infty}(\Td)}+\log n$ by construction of $\hat{f}_{\Phi}$, while 
the first term can be bounded as
\begin{align*}
 \max_{\omega\in\Omega_n}\big|\langle \phi_{\omega},\hat{f}_{\Phi}-f\rangle\big|&\leq \underbrace{\max_{\omega\in\Omega_n}\big|\langle \phi_{\omega},\hat{f}_{\Phi}\rangle-Y_{\omega}\big|}_{\leq\gamma_n}+\max_{\omega\in\Omega_n}\big|\langle \phi_{\omega},f\rangle- Y_{\omega}\big|
 \\
 &\leq \gamma_n+\max_{\omega\in\Omega_n}\frac{\sigma}{\sqrt{n}}\bigg|\int_{\Td}\phi_{\omega}(x)\, dW(x)\bigg|\leq 2\gamma_n
\end{align*}
conditionally on $A_n$, where in the second inequality we used the definition of $\hat{f}_{\Phi}$. This completes the proof of $(i)$. 
For $(ii)$, we have
\begin{equation*}
 \|\hat{f}_{\Phi}-f\|_{\BV}\leq \|\hat{f}_{\Phi}-f\|_{L^{1}(\Td)}+|\hat{f}_{\Phi}-f|_{\BV}\leq \|\hat{f}_{\Phi}-f\|_{L^{\infty}(\Td)}+|\hat{f}_{\Phi}-f|_{\BV}.
\end{equation*}
The first term is bounded by $\|f\|_{L^{\infty}(\Td)}+\log n$, while the second is bounded by $|\hat{f}_{\Phi}|_{\BV}+|f|_{\BV}$. 
Finally, conditionally on $A_n$ we have $|\hat{f}_{\Phi}|_{\BV}\leq |f|_{\BV}$. This is so because $\hat{f}_{\Phi}$ is defined 
as the minimizer of the bounded variation seminorm among the functions satisfying 
$\underset{\omega\in\Omega_n}{\max}|\langle \phi_{\omega},g\rangle-Y_{\omega}|\leq\gamma_n$. Note that, conditionally on $A_n$, 
the function $f$ satisfies this constraint, and hence $f$ is an admissible function for the minimization problem defining 
$\hat{f}_{\Phi}$, whence $|\hat{f}_{\Phi}|_{\BV}\leq|f|_{\BV}$. This completes the proof.
\end{proof}

The proof of Theorem~\ref{Res:Main_thm} relies heavily on results from the theory of function spaces. In particular, we use the following interpolation inequalities.
\begin{proposition}[Interpolation inequalities]\label{InterpInequality}
\quad
 \begin{itemize}
  \item[a)] For $d=1$ and $q\in[1,3]$, there is a constant $C>0$ such that
  \begin{equation*}
   \|g\|_{L^q}\leq C\, (\log n)\, \|g\|_{B^{-1/2}_{\infty,\infty}}^{2/3}\|g\|_{BV}^{1/3}+C\, n^{-1}\, \|g\|_{L^{\infty}}^{2/3}\|g\|_{BV}^{1/3}
  \end{equation*}
holds for any $n\in\mathbb{N}$ and any $g\in L^{\infty}\cap BV(\mathbb{T}^d)$.
\item[b)] Let $d\geq 2$ and $q\in\big[1,\frac{d+2}{d}\big]$. Then there is a constant $C>0$ such that
 \begin{equation*}
  \|g\|_{L^q}\leq C \|g\|_{B^{-d/2}_{\infty,\infty}}^{\frac{2}{d+2}}\|g\|_{BV}^{\frac{d}{d+2}}
 \end{equation*}
holds for any $g\in B^{-d/2}_{\infty,\infty}\cap BV(\mathbb{T}^d)$.
 \end{itemize}
\end{proposition}

 We give the proof of Proposition~\ref{InterpInequality} in Section~\ref{SM:InterpInequ} of the Supplement. It is based on the generalization to periodic functions of a result by~\cite{cohen2003harmonic}, which proves interpolation results between $BV$ and Besov spaces (see Section~\ref{SM:InterpInequ} of the Supplement for the details). The different results in $d=1$ and $d\geq 2$ in Proposition~\ref{InterpInequality} are due to the nature of certain embeddings between Besov and $L^q$ spaces. In a nutshell, interpolation theory allows us to bound the Besov $B^0_{q^*,q^*}$-risk for $q^*=1+2/d$ by the desired rate. In order to translate this bound to the $L^{q^*}$ risk, we use the embedding $B^0_{q,q}\hookrightarrow L^q$, which holds for $q\in(1,2]$ only. This is satisfied for $d\geq 2$, since then $q^*\leq 2$. On the other hand, for $d=1$ we have $q^*=3$, and an alternative strategy has to be applied. We refer to Section~\ref{SM:InterpInequ} of the Supplement for more details.

\begin{proof}[Proof of part a) of Theorem~\ref{Res:Main_thm}]
We prove the claim of part a) Theorem~\ref{Res:Main_thm} conditionally on the event 
 \begin{equation}
A_n:=\bigg\{\max_{\omega\in\Omega_n}\bigg|\int_{\mathbb{T}^d}\phi_{\omega}(x)\, dW(x)\bigg|\leq \frac{\sqrt{n}}{\sigma}\, \gamma_n\bigg\}.\label{Sketch:GoodEvent}
 \end{equation}
 By the choice of $\gamma_n$ in~\eqref{Res:UnivGamma} and part b) of Remark~\ref{Res:RemNoise}, we have
 \begin{equation*}
  \mathbb{P}\big(A_n\big)\geq 1-\big(\#\Omega_n\big)^{1-\kappa^2},
 \end{equation*}
which tends to one as $n\rightarrow\infty$.
\\
Consider first the case $q\leq 1+2/d$. For $d\geq 2$, part b) of Proposition~\ref{InterpInequality} applies and gives the interpolation 
inequality
 \begin{equation*}
   \|\hat{f}_{\Phi}-f\|_{L^q(\mathbb{T}^d)}\leq C\|\hat{f}_{\Phi}-f\|_{\Besov}^{\frac{2}{d+2}}\|\hat{f}_{\Phi}-f\|_{\BV}^{\frac{d}{d+2}}.
 \end{equation*}
Conditionally on $A_n$, Proposition~\ref{App:aux_Main_thm} gives us bounds for the terms in the 
right-hand side, which inserted give 
\begin{align*}
 \|\hat{f}_{\Phi}-f\|_{L^q(\mathbb{T}^d)}&\leq C \bigg(\gamma_n+C\frac{\|f\|_{L^{\infty}(\Td)}+\log n}{\sqrt{n}}\bigg)^{\frac{2}{d+2}}\big(\|f\|_{L^{\infty}(\Td)}+2|f|_{\BV}+\log n\big)^{\frac{d}{d+2}}
 \\
 &\leq Cn^{-\frac{1}{d+2}}\big(\sqrt{\log \#\Omega_n}+L+\log n\big)^{\frac{2}{d+2}}\big(L+\log n\big)^{\frac{d}{d+2}}
 \\
 &\leq C\, n^{-\frac{1}{d+2}}\, \log n
\end{align*}
using that $f\in BV_L$. Since $\#\Omega_n\leq Q(n)$ grows at most polynomially in $n$, the claim follows.

For the case $d=1$, we use part a) of Proposition~\ref{InterpInequality}, which yields 
\begin{equation*}
\|g\|_{L^q}\leq C\, (\log n)\, \|g\|_{B^{-1/2}_{\infty,\infty}}^{2/3}\|g\|_{BV}^{1/3}+C\, n^{-1}\, \|g\|_{L^{\infty}}^{2/3}\|g\|_{BV}^{1/3}
\end{equation*}
for $g=\hat{f}_{\Phi}-f$ and $q\in[1,3]$. Proposition~\ref{App:aux_Main_thm} now implies that, conditionally on $A_n$, we have
\begin{equation*}
\|\hat{f}_{\Phi}-f\|_{L^q}\leq C\, n^{-1/3}\, (\log n)^2 +C\, n^{-1}\, \log n,
\end{equation*}
which yields the claim.

We have proved the claim for the $L^q$-risk with $q\leq 1+2/d$. For larger $q$, we use H\"older's inequality between the $L^{1+2/d}$ and the $L^{\infty}$-risk, which gives the desired bound.
\end{proof}

\begin{proof}[Proof of part b) of Theorem~\ref{Res:Main_thm}]
 Using the convergence conditionally on $A_n$ proved in part a), we can bound the expected risk as
 \begin{align}
  \mathbb{E}[\|\hat{f}_{\Phi}-f\|_{L^q(\mathbb{T}^d)}]&=\mathbb{E}[\|\hat{f}_{\Phi}-f\|_{L^q(\mathbb{T}^d)}\, 1_{A_n}] +\mathbb{E}[\|\hat{f}_{\Phi}-f\|_{L^q(\mathbb{T}^d)}\, 1_{A_n^c}]\nonumber
 \\
 &\leq C\, r_n\, \mathbb{P}\big(A_n\big)+\mathbb{E}[\|\hat{f}_{\Phi}-f\|_{L^q(\mathbb{T}^d)}\, 1_{A_n^c}]\nonumber
 \\
 &\leq C\, r_n+\mathbb{E}[\|\hat{f}_{\Phi}-f\|_{L^q(\mathbb{T}^d)}\, 1_{A_n^c}],\label{pf:bound1Mul}
 \end{align}
where $r_n= n^{-\min\{\frac{1}{d+2},\frac{1}{dq}\}}\, (\log n)^{3-\min\{d,2\}}$. The rest of the proof consists in showing that the second term behaves as 
$o(n^{-1/2})$ for $\kappa^2>1+\frac{1}{(d+2)\Gamma}$. By assumption we have the bounds $\|f\|_{L^{\infty}}\leq L$ and $\|\hat{f}_{\Phi}\|_{L^{\infty}}\leq \log n$, so we can bound the second term as
\begin{align*}
\mathbb{E}[\|\hat{f}_{\Phi}-f\|_{L^q(\mathbb{T}^d)}\, 1_{A_n^c}] &\leq \mathbb{E}[\|\hat{f}_{\Phi}-f\|_{L^{\infty}(\mathbb{T}^d)}\, 1_{A_n^c}] \leq (L+\log n)\mathbb{P}\big(A_n^c\big).
\end{align*}
By part b) of Remark~\ref{Res:RemNoise} we have $\mathbb{P}(A_n^c)\leq (\#\Omega)^{1-\kappa^2}$, and inserting this back in~\eqref{pf:bound1Mul} yields
\begin{equation*}
\mathbb{E}[\|\hat{f}_{\Phi}-f\|_{L^q(\mathbb{T}^d)}]\leq C\,  n^{-\min\{\frac{1}{d+2},\frac{1}{dq}\}}\, (\log n)^{3-\min\{d,2\}}+C\, n^{-\Gamma(\kappa^2-1)}\log n.
\end{equation*}
 Choosing $\kappa^2>1+1/((d+2)\Gamma)$ yields the claim.
 \end{proof}

\subsection{Minimax rate over $BV$}\label{SM:minimax}

Here we prove Theorem~\ref{MinimaxThm} by showing a lower bound for the minimax risk over Besov spaces $B^1_{1,1}$ with respect to the $L^q$-risk. This implies a lower bound for the minimax risk over $BV_L$, since
\begin{equation*}
BV_L\supset (B^{1}_{1,1}\cap L^{\infty})_L:=\big\{g\in B^1_{1,1}\, \big|\, \|g\|_{B^1_{1,1}}\leq L, \ \ \|g\|_{L^{\infty}}\leq L\big\}.
\end{equation*}
The minimax $L^q$-risk for $q\leq 1+2/d$ (dense case) is well understood, and the associated minimax rates have been known for a while to be $n^{-\frac{1}{d+2}}$. Its proof follows the classical strategy of constructing a set of alternatives in $(B^{1}_{1,1}\cap L^{\infty})_L$ that are well separated in the $L^q$-norm, and applying an information inequality (e.g.~Fano's inequality). It can be found in Chapter 10 of~\cite{hardle2012wavelets}, so we do not reproduce it here.

On the other hand, the regime $q\geq 1+2/d$ is far less popular, and we have not found any proof of what the minimax rate is there. The difficulty here is that $B^1_{1,1}$ is a Besov space with "$s\leq d/p$", and the literature has focused mainly on the case $s>d/p$ (with some exceptions, see~\cite{goldenschluger} and~\cite{lepskii2015}). Our proof that the minimax rate is $O(n^{-\frac{1}{dq}})$ in that regime follows the same idea as in the other regimes: we construct a set of well separated alternatives and show that no statistical procedure can distinguish them perfectly. As in the dense regime, our construction is based on Assouad's cube~\citep{assouad1983deux}.

\begin{proof}[Proof of Theorem~\ref{MinimaxThm}]

Our proof follows the proof of Theorem 10.3 in~\cite{hardle2012wavelets} closely. We structure it in several steps.
\\

\textbf{Construction of alternatives: } Let $g_0\in B^1_{1,1}\cap L^{\infty}$ satisfy
\begin{equation*}
\|g_0\|_{B^{1}_{1,1}}\leq L/2,\ \ \textup{ and } \ \ \|g_0\|_{L^{\infty}}\leq L/2.
\end{equation*}
Let $\psi_{j,k,e}$ be a basis of Daubechies wavelets with $S$ continuous partial derivatives, where $S>\max\{1,d/2\}$.
For $j\geq 0$ to be fixed later, let $R_j\subseteq \{0,\ldots,2^j-1\}^d\times E_j$ denote a subset of wavelet indices such that
\begin{equation*}
\textup{supp }\psi_{j,k,e}\cap\textup{supp }\psi_{j,k',e'}=\emptyset \ \ \textup{ for } (k,e)\neq (k',e')\in R_j.
\end{equation*}
Since Daubechies wavelets are compactly supported, we have $\# R_j\leq c2^{jd}$ for a constant $c>0$. 
Let further $S_j=\# R_j=\lfloor 2^{j\Delta}\rfloor$ for a real number $\Delta\in[0,d]$ to be chosen later. 
Consider now vectors $\epsilon\in\{-1,+1\}^{S_j}$ with components indexed by $(k,e)\in R_j$. Our alternatives will have the form
\begin{equation*}
g^{\epsilon}:=g_0+\gamma\sum_{(k,e)\in R_j}\epsilon_{k,e}\psi_{j,k,e}
\end{equation*}
for $\gamma>0$ to be chosen later. 
Define the set $\mathcal{G}:=\{g^{\epsilon}\, |\, \epsilon\in \{-1,+1\}^{S_j}\}$. Notice that all functions in this set satisfy
\begin{equation*}
\|g^{\epsilon}\|_{B^{1}_{1,1}}\leq L \ \ \textup{ and } \ \  \|g^{\epsilon}\|_{L^{\infty}}\leq L
\end{equation*}
provided that
\begin{equation}
\gamma\leq \frac{L}{2}\, 2^{-j(1-d/2+\Delta)} \ \ \textup{ and }\ \ \gamma\leq \frac{L}{2\, \|\psi\|_{L^{\infty}}}\, 2^{-jd/2},\label{jBalance}
\end{equation}
respectively. In the following we choose $\Delta=d-1$ in order to balance these two terms. Finally, the $L^q$-separation between these alternatives is
\begin{align}
\delta:=\inf_{\epsilon\neq\epsilon'}\|g^{\epsilon}-g^{\epsilon'}\|_{L^q}&=\|\gamma\psi_{j,k,e}\|_{L^q}=\gamma\, 2^{jd(\frac{1}{2}-\frac{1}{q})}\, \|\psi\|_{L^q}\label{Assouaddelta},
\end{align}
where the first equality follows from the disjoint supports of the wavelets. 

\textbf{Lower bound: } We use now Assouad's lemma for lower bounding the $L^q$-risk over $(B^1_{1,1}\cap L^{\infty})_L$. We reproduce the claim (Lemma 10.2 in~\cite{hardle2012wavelets}) for completeness. 
\begin{lemma}\label{lemma:Assouad}
For $\epsilon\in\{-1,+1\}^{S_j}$ and $(k,e)\in R_j$, define $\epsilon_{*k}:=(\epsilon_{(k_1,e_1)}',\ldots,\epsilon_{(k_{S_j},e_{S_j})}')$, where
\begin{equation*}
\epsilon_{(k'e')}'=\begin{cases}
\epsilon_{(k,e)}\ & \textup{ if } (k',e')\neq (k,e),
\\
-\epsilon_{(k,e)}\ & \textup{ if } (k',e')= (k,e).
\end{cases}
\end{equation*}
Assume there exist constants $\lambda,p_0>0$ such that
\begin{equation}
\mathbb{P}_{g^{\epsilon}}\big(LR(g^{\epsilon_{*k}},g^{\epsilon})>e^{-\lambda}\big)\geq p_0, \ \ \forall\epsilon, n,\label{eq:lemmaAssouad}
\end{equation}
where $\mathbb{P}_{g^{\epsilon}}$ denotes the probability with respect to observations drawn from $g^{\epsilon}$ in the white noise model, and $LR(g^{\epsilon_{*k}},g^{\epsilon})$ denotes the likelihood ratio between the observations associated to $g^{\epsilon_{*k}}$ and $g^{\epsilon}$. Then any estimator $\hat{f}$ satisfies
\begin{equation*}
\sup_{g^{\epsilon}\in \mathcal{G}}\mathbb{E}_{g^{\epsilon}}\|\hat{f}-g^{\epsilon}\|_{L^q}\geq \frac{e^{-\lambda}\, p_0}{2}\, \delta\, S_j^{1/q},
\end{equation*}
where $\delta$ is defined in~\eqref{Assouaddelta}.
\end{lemma}

\textbf{Verification of~\eqref{eq:lemmaAssouad}: }
The condition~\eqref{eq:lemmaAssouad} is easily verified in our setting with Gaussian observations under the condition that $n\gamma^2\leq c$ for $n$ large enough (see Section 10.5 in~\cite{hardle2012wavelets}). Indeed, by Markov's inequality we have
\begin{equation*}
\mathbb{P}_{g^{\epsilon}}\big(LR(g^{\epsilon_{*k}},g^{\epsilon})>e^{-\lambda}\big)\geq 1-\frac{1}{\log e^{\lambda}}\mathbb{E}_{g^{\epsilon}}\bigg|\log LR(g^{\epsilon_{*k}},g^{\epsilon})\bigg|,
\end{equation*} 
and using Proposition 6.1.7 in~\cite{gine2015mathematical} to bound the expectation by the Kullback-Leibler divergence we get
\begin{equation*}
\mathbb{P}_{g^{\epsilon}}\big(LR(g^{\epsilon_{*k}},g^{\epsilon})>e^{-\lambda}\big)\geq 1-\frac{1}{\lambda}\bigg(K(dP_{g^{\epsilon_{*k}}},dP_{g^{\epsilon}})+\sqrt{2 K(dP_{g^{\epsilon_{*k}}},dP_{g^{\epsilon}})}\bigg).
\end{equation*}
Using the Cameron-Martin Theorem to interpret the Gaussian probability measures (see Theorem 2.6.13 in~\cite{gine2015mathematical}), the Kullback-Leibler divergence between Gaussian measures is easily computer and gives
\begin{equation*}
K(dP_{g^{\epsilon_{*k}}},dP_{g^{\epsilon}})=\frac{n}{2\sigma^2}\|g^{\epsilon_{*k}}-g^{\epsilon}\|_{L^2}^2=\frac{n\gamma^2}{2\sigma^2}\|\psi_{j,k,e}\|_{L^2}^2=\frac{n\gamma^2}{2\sigma^2}.
\end{equation*}
Hence, choosing $\gamma=t_0\, n^{-1/2}$ for a small enough constant $t_0>0$ gives~\eqref{eq:lemmaAssouad}.

\textbf{Application of Lemma~\ref{lemma:Assouad}: } The conclusion of the lemma applies, and we can lower bound the $L^q$-risk over the class $(B^1_{1,1}\cap L^{\infty})_L$ by the risk over $\mathcal{G}$, i.e.,
\begin{equation}
\sup_{f\in (B^1_{1,1}\cap L^{\infty})_L}\mathbb{E}_f\|\hat{f}-f\|_{L^q}\geq \sup_{g^{\epsilon}\in \mathcal{G}}\mathbb{E}_{g^{\epsilon}}\|\hat{f}-g^{\epsilon}\|_{L^q}\geq \frac{e^{-\lambda}\, p_0}{2}\, \delta\, 2^{j\Delta/q}\label{LowRisk1}
\end{equation}
for any estimator $\hat{f}$. 
It remains to choose the scale parameter $j\geq 0$. Recall that we have chosen $\gamma=t_0\, n^{-1/2}$. Further, by~\eqref{jBalance} we also need $\gamma\leq c\, 2^{-j(1-d/2+\Delta)}=c\, 2^{-jd/2}$, for the choice $\Delta=d- 1$. We choose $j$ such that $2^{-jd/2}=c\, n^{-1/2}$, which gives the bound in~\eqref{LowRisk1}
\begin{align*}
\delta\, 2^{j\Delta/q}=c\, \gamma\, 2^{jd(\frac{1}{2}-\frac{1}{q})}\, 2^{j\Delta/q}&=c\, \bigg(\frac{1}{n}\bigg)^{\frac{1}{2}-\big(\frac{1}{2}-\frac{1}{q}\big)-\frac{\Delta}{dq}}
=c\, n^{-\frac{1}{dq}}.
\end{align*}
This completes the proof.
\end{proof}

\section{Summary and outlook}\label{Sect:Summary}

We presented a family of estimators in the Gaussian white noise model defined by minimization of the $BV$-seminorm under a 
constraint on the frame coefficients of the residuals. Under conditions on the frame that amount to a certain compatibility with the Besov space $\besov$, we show that these estimators attain the minimax optimal 
rate of convergence in any dimension up to logarithmic factors. There are still several open questions regarding extensions of 
our estimator. First, the extension to a nonparametric regression model with discretely sampled data, which would involve a discretization of the inner 
products $\langle \phi_{\omega},f\rangle_{L^2}$. This discretization induces an error of the order $O(n^{-1/d})$ 
in the approximation of the Besov norm by the maximum of the frame coefficients of a function, which results in slower 
convergence rates of the form $\|\hat{f}_{\Phi}-f\|_{L^q}\leq C\, n^{-\min\{\frac{1}{d+2},\frac{1}{dq}\}\min\{1,2/d\}}\textup{Polylog}_d(n)$. 
In dimensions $d=1,2$ the polynomial rate equals $n^{-\min\{\frac{1}{d+2},\frac{1}{dq}\}}$, which coincides with the minimax rate over the class $BV_L$ up to logarithmic factors. In dimension $d\geq 3$, the discretization error dominates and the polynomial rate is $n^{-\frac{2}{d(d+2)}}$ for $q\leq 1+2/d$, and $n^{-\frac{1}{dq}}$ for $q>1+2/d$. 
We do not know whether this rate is sharp in a minimax sense (up to logarithmic factors). Notice that the asymptotic equivalence 
of the white noise and the multivariate nonparametric regression models derived by~\cite{reiss2008} does not apply for functions 
of bounded variation, so the minimax rates need not be the same in the two models. We leave the clarification of this question for future research.

A second question concerns the relation between the multiscale data-fidelity and statistical testing. In fact, our use of dictionary elements with $L^2$-norm equal to one is analogous to the multiplicative scaling used by~\cite{dumbgen2001} to correctly weight their multiresolution test statistics. This raises the question of whether an \textit{additive scaling} in our data-fidelity is necessary in our setting, as it is in theirs. The answer is that such an additive scaling would help us remove some (but not all) of the logarithmic terms in the error bound in Theorem~\ref{Res:Main_thm}. However, it would imply additional difficulties in the theoretical analysis of the estimator, since the constraint would no longer match the Besov scale exactly. Alternatively, a different multiplicative scaling could be used to link the multiscale data-fidelity with the \textit{logarithmic} Besov spaces (see Section 4.4 in~\cite{gine2015mathematical}). We leave as an open question whether these modified data-fidelities and Besov spaces could yield an improved performance.

Another interesting question concerns the choice of the risk functional. We have proven convergence rates with respect to the $L^q$-risk, which measures the \textit{global} error made by the estimator. In contrast, the use of multiscale risk functionals has been proposed as an alternative quality measure which takes spatial adaptation into account (see e.g.~\cite{cai2005nonparametric} and~\cite{li2016variational}). We expect that estimators of the form~\eqref{Intro:TV_est1} should perform particularly well with respect to such multiscale risks, and postpone the answer to that question for future work.

The extension of our theory to statistical inverse problems is particularly attractive, since in many applications one only has 
access to a transformed version of the object of interest (see e.g.~\cite{frick2013statistical} 
and~\cite{niinimaki2016multiresolution} for applications of TV-regularization to microscopy and tomography, respectively). 
The analysis done in the present paper is expected to be adaptable to inverse 
problems if the operator is assumed to have ``good'' mapping properties in the Besov scale $B^{s}_{\infty,\infty}$. 
The modification would essentially involve a constraint of the form 
$\max_{\omega\in\Omega_n}\big|\langle \phi_{\omega},Tg\rangle -Y_{\omega}\big|\leq\gamma_n$ in~\eqref{Intro:TV_est1}, 
where $T$ is the forward operator (see~\cite{frick2013statistical} and~\cite{li2016variational} for examples and analysis of such an estimator). 
From this constraint it is apparent that the dictionary $\Phi$ has to depend on the forward operator $T$ 
(see~\cite{proksch2016} for a similar construction). 
Finally, the extension to nongaussian noise models is of interest in many applications. In that respect, note that the analysis of the 
estimator~\eqref{Intro:TV_est1} depends on the tail behavior of the statistic 
$\max_{\omega\in\Omega_n}|\langle \phi_{\omega},dW\rangle|$ being \textit{subgaussian}. Finally, the extension to SDE-based models (see e.g.~\cite{gobet2004}) appears to us of interest.

\section{Appendix}\label{SM:section}

This Appendix is organized as follows. In Section~\ref{SM:InterpInequ} we prove the interpolation inequalities of Proposition~\ref{InterpInequality}, and in Section~\ref{SM:Proof_Diction} we prove Propositions~\ref{Res:Veri_Wavelet},~\ref{Res:veri_multi} and~\ref{Res:Veri_Curv}

\subsection{Interpolation inequalities}\label{SM:InterpInequ}

Here we prove the interpolation inequalities in Proposition~\ref{InterpInequality} in the main text, which are based on the following interpolation result.

\begin{proposition}\label{Res:InterProp}
 Let $s\in\mathbb{R}$ and $p\in(1,\infty]$. Let $\gamma=1+(s-1)p'/d$ be such that $\gamma<1-1/d$ or $\gamma>1$, where $p'$ is 
 the H\"older conjugate of $p$. Then for any $\vartheta\in(0,1)$ and parameters such that
\begin{equation*}
 \frac{1}{q}=\frac{1-\vartheta}{p}+\vartheta,\ \ \ t=(1-\vartheta)s+\vartheta,
\end{equation*}
we have
\begin{equation}
 \|g\|_{B^t_{q,q}}\leq C\|g\|_{B^s_{p,p}}^{1-\vartheta}\|g\|_{BV}^{\vartheta}\label{Res:GenIn}
\end{equation}
for any $g\in B^s_{p,p}\cap BV(\mathbb{T}^d)$.
\end{proposition}

Proposition~\ref{Res:InterProp} is a generalization to periodic functions of a result by~\cite{cohen2003harmonic}. Its proof is based on a refined analysis of the wavelet coefficients of $BV$ functions, that are shown to belong to weak weighted $\ell^p$ spaces. Alternatively, an independent proof by~\cite{ledoux2003} is based on the thermic representation of Besov spaces and on pseudo-Poincar\'e inequalities for the semigroup associated with that thermic representation. 
Since the adaptation of the proof from~\cite{cohen2003harmonic} to the periodic case does not involve any novel ideas, we omit it. In order to derive part b) of Proposition~\ref{InterpInequality} from Proposition~\ref{Res:InterProp}, we choose $s=-d/2$, $t=0$, $p=\infty$ and $q=(d+2)/d$. Then for $d\geq 2$, the norm in the left-hand side of~\eqref{Res:GenIn} can be readily reformulated in terms of an $L^q$ norm. For $d=1$ the situation is more involved, since the embedding $B^0_{3,3}\hookrightarrow L^3$ does \textit{not} hold. A more refined argument is needed to prove convergence in $L^q$ for $d=1$, for which we use a variation of part a) of Proposition~\ref{InterpInequality}. This difference is responsible for the different logarithmic factors in~\eqref{Res:Gen_Conv}.

\begin{proposition}\label{ParticInt}
 Let $d\geq 2$ and $q\in\big[1,\frac{d+2}{d}\big]$. Then there is a constant $C>0$ such that
 \begin{equation*}
  \|f\|_{L^q(\mathbb{T}^d)}\leq C \|f\|_{B^{-d/2}_{\infty,\infty}(\mathbb{T}^d)}^{\frac{2}{d+2}}\|f\|_{BV(\mathbb{T}^d)}^{\frac{d}{d+2}}
 \end{equation*}
holds for any $f\in B^{-d/2}_{\infty,\infty}(\mathbb{T}^d)\cap BV(\mathbb{T}^d)$.
\end{proposition}
\begin{proof}
  In the notation of Proposition~\ref{Res:InterProp}, the choice $s=-d/2$ and $p=\infty$ yields $\gamma=\frac{1}{2}-\frac{1}{d}<1-\frac{1}{d}$, 
  so Proposition~\ref{Res:InterProp} applies and yields for $t=0$ the inequality
  \begin{equation}
   \|f\|_{B^0_{\frac{d+2}{d},\frac{d+2}{d}}(\mathbb{T}^d)}\leq C \|f\|_{B^{-d/2}_{\infty,\infty}(\mathbb{T}^d)}^{\frac{2}{d+2}}\|f\|_{BV(\mathbb{T}^d)}^{\frac{d}{d+2}}.\label{IntermediateInterp}
  \end{equation}
Now, Remark 4 in Section 3.5.1 of~\cite{SchmeisserTriebel} gives
\begin{equation*}
 B_{r,r}^0(\mathbb{T}^d)=F^0_{r,r}(\mathbb{T}^d)\subset F_{r,2}^0(\mathbb{T}^d),
\end{equation*}
where the last embedding is continuous and holds for $0<r\leq 2$, and $F^s_{p,q}(\mathbb{T}^d)$ is a Triebel-Lizorkin space 
(see~\cite{SchmeisserTriebel}). Finally, we use that
\begin{equation*}
F_{r,2}^0(\mathbb{T}^d)=L^r(\mathbb{T}^d)
\end{equation*}
holds for any $1<r<\infty$ (see Remark 2 of Section 3.5.4 in~\cite{SchmeisserTriebel}). 
These embeddings give the inequality
\begin{equation*}
 \|f\|_{L^r(\mathbb{T}^d)}\leq C\|f\|_{B_{r,r}^0(\mathbb{T}^d)}
\end{equation*}
for any $r\in(1,2]$ and a constant $C>0$ independent of $f$. Hence, for $q\in\big[1,\frac{d+2}{d}\big]$, $d\geq 2$, we have
\begin{equation*}
 \|f\|_{L^q(\mathbb{T}^d)}\leq \|f\|_{L^{\frac{d+2}{d}}(\mathbb{T}^d)}\leq C\|f\|_{B_{\frac{d+2}{d},\frac{d+2}{d}}^0(\mathbb{T}^d)},
\end{equation*}
which together with~\eqref{IntermediateInterp} yields the claim.
\end{proof}

Proposition~\ref{ParticInt} gives us a (sharp) interpolation inequality for $d\geq 2$ only. For the one-dimensional case we can nevertheless 
derive a slightly weaker result, which in the proof of convergence rates still gives the right polynomial rate of convergence but 
some additional logarithmic terms.

\begin{proposition}\label{ParticIntd=1}
 Let $d=1$ and $q\in\big[1,3\big]$. Then there is a constant $C>0$ such that for any $n\in\mathbb{N}$ we have
 \begin{equation*}
  \|f\|_{L^q}\leq C (\log n )\, \|f\|_{B^{-d/2}_{\infty,\infty}}^{2/3}\|f\|_{BV}^{1/3}+C\, n^{-1}\, \|f\|_{L^{\infty}}^{2/3}\, \|f\|_{BV}^{1/3}
 \end{equation*}
 for any $f\in L^{\infty}(\mathbb{T}^d)\cap BV(\mathbb{T}^d)$.
\end{proposition}

\begin{proof}
The claim follows from Proposition~\ref{Res:InterProp} with $s=-1/2$ and $p=\infty$, which gives a bound on the $B^0_{3,3}$ norm. The $L^q$-norm, $q\in[1,3]$, can be upper bounded by the $L^3$-norm, which itself can be upper bounded by the $B^0_{3,3}$ norm using Proposition~\ref{new_d=1} below. Choosing $J=3\log n$ yields the claim.
\end{proof}

\begin{proposition}\label{new_d=1}
Let $g\in L^{\infty}\cap BV$. Then for any $J\in\mathbb{N}$ we have
\begin{equation*}
\|g\|_{L^3}\leq C\,  J\, \|g\|_{B^0_{3,3}}+C\, 2^{-J/3}\|g\|_{L^{\infty}}^{2/3}\|g\|_{BV}^{1/3}
\end{equation*}
for a constant $C>0$ independent of $g$.
\end{proposition}

Before we prove Proposition~\ref{new_d=1} we give a technical lemma concerning wavelet series.

\begin{lemma}\label{lemma:wavelets}
Let $\{\psi_{j,k,e}\}$ denote a basis of compactly supported wavelets in $L^2(\Td)$. There is a constant $C_{\psi}$ such that
\begin{equation*}
\int_{\Td}\bigg|\sum_{(k,e)\in P_j^d\times E_j}c_{j,k,e}\psi_{j,k,e}(x)\bigg|^3\, dx\leq C_{\psi}\, 2^{j3d(1/2-1/3)}\sum_{(k,e)\in P_j^d\times E_j}|c_{j,k,e}|^3
\end{equation*}
for any $j\in\mathbb{N}$ and any coefficients $\{c_{j,k,e}\}$.
\end{lemma}
\begin{proof}
 Due to the compact support of the wavelets, there is a constant $c_{\psi}$ such that, for each $j\geq 0$ and $(k,e)\in\{0,\ldots,2^{j}-1\}^d\times E_j$, at most $c_{\psi}$ wavelets have support intersecting the support of $\psi_{j,k,e}$, i.e.,
 \begin{equation*}
\max_{(j,k,e)\in\mathbb{N}\times  P_j^d\times E_j}\#\mathcal{I}_{j,k,e}\leq c_{\psi}
 \end{equation*}
 where
 \begin{equation*}
\mathcal{I}_{j,k,e}:=\big\{(k',e')\in P_j^d\times E_j\, \big|\, \textup{ supp }\psi_{j,k,e}\cap\textup{ supp } \psi_{j,k',e'}\neq \emptyset\big\}.
 \end{equation*}
As a consequence, we have the following inequalities
\begin{align*}
\int_{\Td}\bigg|\sum_{(k,e)\in P_j^d\times E_j}c_{j,k,e}&\psi_{j,k,e}(x)\bigg|^3\, dx=\sum_{(k,e)\in P_j^d\times E_j}\int_{\Td}\big|c_{j,k,e}\psi_{j,k,e}(x)\big|^3\, dx
\\
&+3\sum_{(k,e)\neq (k',e')}\int_{\Td}\big|c_{j,k,e}\psi_{j,k,e}(x)\big|^2\big|c_{j,k',e'}\psi_{j,k',e'}(x)\big|\, dx
\\
& \hspace{-3cm}+6\sum_{(k,e)\neq (k',e')\neq (k'',e'')}\int_{\Td}\big|c_{j,k,e}\psi_{j,k,e}(x)\big|\big|c_{j,k',e'}\psi_{j,k',e'}(x)\big|\big|c_{j,k'',e''}\psi_{j,k'',e''}(x)\big|\, dx
\\
&\leq (1+3c_{\psi}+6c_{\psi}^2)\sum_{(k,e)\in P_j^d\times E_j}|c_{j,k,e}|^3\|\psi_{j,k,e}\|_{L^3}^3
\\
&= (1+3c_{\psi}+6c_{\psi}^2)\, \|\psi\|_{L^3}^3\, 2^{j3d(1/2-1/3)}\sum_{(k,e)\in P_j^d\times E_j}|c_{j,k,e}|^3
\end{align*}
where in the last equality we used that $\|\psi_{j,k,e}\|_{L^3}=2^{jd(1/2-1/3)}\, \|\psi\|_{L^3}$. The inequality is justified as follows. By Young's inequality and the support properties of $\psi_{j,k,e}$ we have
\begin{align*}
\sum_{(k,e)\neq (k',e')}\int_{\Td}&\big|c_{j,k,e}\psi_{j,k,e}(x)\big|^2\big|c_{j,k',e'}\psi_{j,k',e'}(x)\big|\, dx 
\\
&\leq 
\sum_{(k,e)\neq (k',e'),\  (j,k',e')\in\mathcal{I}_{j,k,e}}\int_{\Td}\frac{2}{3}\big|c_{j,k,e}\psi_{j,k,e}(x)\big|^3+\frac{1}{3}\big|c_{j,k',e'}\psi_{j,k',e'}(x)\big|^3\, dx
\\
&\leq \frac{2}{3}\, c_{\psi}\sum_{(k,e)}\int_{\Td}\big|c_{j,k,e}\psi_{j,k,e}(x)\big|^3\, dx+  \frac{1}{3}\, c_{\psi}\sum_{(k',e')}\int_{\Td}\big|c_{j,k',e'}\psi_{j,k',e'}(x)\big|^3\, dx
\\
&=c_{\psi}\sum_{(k,e)}\big|c_{j,k,e}\big|^3\|\psi_{j,k,e}\|_{L^3}^3.
\end{align*}
The same argument gives the desired bound for the product of three terms. This completes the proof.
\end{proof}

\begin{proof}[Proof of Proposition~\ref{new_d=1}]
Let $\{\psi_{j,k,e}\}$ be a basis of compactly supported wavelets. Writing $g$ formally as its wavelet series we have
\begin{equation}
\|g\|_{L^3}=\bigg\|\sum_{j\in\mathbb{N}}\sum_{k,e}c_{j,k,e}\psi_{j,k,e}\bigg\|_{L^3}\leq \bigg\|\sum_{j\leq J}\sum_{k,e}c_{j,k,e}\psi_{j,k,e}\bigg\|_{L^3}+\bigg\|\sum_{j>J}\sum_{k,e}c_{j,k,e}\psi_{j,k,e}\bigg\|_{L^3}\label{aux:d1_first}
\end{equation}
for any $J\in\mathbb{N}$. Using Lemma~\ref{lemma:wavelets}, the first term can be bounded as
\begin{align*}
\bigg\|\sum_{j\leq J}\sum_{k,e}c_{j,k,e}\psi_{j,k,e}\bigg\|_{L^3}&\leq \sum_{j\leq J}\bigg(C_{\psi}2^{j3d(1/2-1/3)}\sum_{(k,e)}|c_{j,k,e}|^3\bigg)^{1/3}
\\
&\leq C_{\psi}^{1/3}\, J\, \bigg(\max_{j\leq J}\, 2^{j3d(1/2-1/3)}\sum_{(k,e)}|c_{j,k,e}|^3\bigg)^{1/3}
\\
&\leq  C_{\psi}^{1/3}\, J\,\|g\|_{B^0_{3,3}},
\end{align*}
which gives the first term of the claim. For the second term, we use that $g\in L^{\infty}$ and $g\in BV$, which means that the wavelet coefficients of $g$ satisfy the bounds
\begin{align*}
\max_{(k,e)\in P_j^d\times E_j}|c_{j,k,e}|&\leq 2^{-jd/2}\, \|g\|_{L^{\infty}}
\\
 \sum_{(k,e)\in P_j^d\times E_j}|c_{j,k,e}|&\leq 2^{j(d/2-1)}\, \|g\|_{BV},
\end{align*}
for any $j\in\mathbb{N}$, where the first inequality follows from the compact support of the wavelets and H\"older's inequality, and the second follows from the embedding $BV\subset B^1_{1,\infty}$. 
Using Lemma~\ref{lemma:wavelets} and these bounds, the second term in~\eqref{aux:d1_first} can be bounded as
\begin{align*}
\bigg\|\sum_{j>J}\sum_{k,e}c_{j,k,e}\psi_{j,k,e}\bigg\|_{L^3}&\leq\sum_{j>J} \bigg(C_{\psi}2^{j3d(1/2-1/3)}\sum_{(k,e)}|c_{j,k,e}|^3\bigg)^{1/3}
\\
&\leq C_{\psi}^{1/3}\sum_{j>J} \bigg(2^{j3d(1/2-1/3)}\, 2^{-jd}\, \|g\|_{L^{\infty}}^22^{j(d/2-1)}\|g\|_{BV}\bigg)^{1/3}
\\
&\leq C_{\psi}^{1/3}\, \|g\|_{L^{\infty}}^{2/3}\, \|g\|_{BV}^{1/3}\sum_{j>J} 2^{-j/3},
\end{align*}
which gives the claim.
\end{proof}

\subsection{Verification of assumptions for particular dictionaries}\label{SM:Proof_Diction}

\subsubsection{Proof of Proposition~\ref{Res:Veri_Wavelet}}\label{SM:VerificationWavelet}

\begin{proof}[Proof of Proposition~\ref{Res:Veri_Wavelet}]
We begin with the inequality in Assumption~\ref{Res:Gen_Ass}. Recall that the Besov norm of a function can be equivalently represented 
in terms 
of its wavelet coefficients with respect to a smooth enough wavelet basis (see Theorem 4.3.26 in~\cite{gine2015mathematical} for 
the one-dimensional case, and Section 1.3.3 in~\cite{Triebel} for the general case). In particular we have
\begin{align*}
 \|g\|_{\Besov}&\asymp \sup_{j\geq 0}\max_{k\in P_j^d}\max_{e\in E_j}|\langle \psi_{j,k,e},g\rangle|
 \\
 &\leq \max_{0\leq j<J}\max_{k\in P_j^d}\max_{e\in E_j}|\langle \psi_{j,k,e},g\rangle|+\sup_{j\geq J}\max_{k\in P_j^d}\max_{e\in E_j}|\langle \psi_{j,k,e},g\rangle|.
\end{align*}
Note that the first term is precisely $\max_{(j,k,e)\in\Omega_n}|\langle \psi_{j,k,e},g\rangle|$ for 
$J=\big\lfloor\frac{1}{d}\frac{\log n}{\log 2}\big\rfloor$ and $\Omega_n$ as in equation~\eqref{Res:Wave_Sets}. It remains to show that the second term 
is dominated by $C \|g\|_{L^{\infty}(\Td)}\, n^{-1/2}$. For that, H\"older's inequality yields
\begin{align}
 \sup_{j\geq J}\max_{k\in P_j^d}\max_{e\in E_j}|\langle \psi_{j,k,e},g\rangle|&\leq  \sup_{j\geq J}\max_{k\in P_j^d}\max_{e\in E_j}\, \|\psi_{j,k,e}\|_{L^1(\Td)}\|g\|_{L^{\infty}(\Td)}\nonumber
 \\
 &\leq C\, 2^{-Jd/2}\|g\|_{L^{\infty}(\Td)},\label{SM:WaveJack}
\end{align}
where we used that the wavelets are of the form $\psi_{j,k,e}(x)=2^{jd/2}\psi_e\big(2^jx-k\big)$. Using now that $2^{-Jd/2}\leq 2^{d/2}\, n^{-1/2}$, the inequality follows.
Morevoer, since the index sets $\Omega_n$ satisfy $2^{-d}n\leq \#\Omega_n\leq n$, we can choose $Q(x)=x$ and $\Gamma=1$ in Assumption~\ref{Res:Gen_Ass}. This completes the proof.
\end{proof}

\subsubsection{Proof of Proposition~\ref{Res:veri_multi}}\label{App:Mu}

It remains to prove Proposition~\ref{Res:veri_multi} for the multiresolution system. For that, we rely on the characterization of Besov spaces in terms of local means. In particular, we use the norm equivalence
\begin{equation}
\|g\|_{\Besov}\asymp \sup_{j\geq 0} 2^{jd/2}\sup_{x\in[0,1)^d}\bigg|\int_{[0,1)^d}K\big(2^j(y-x)\big)g(y)\, dy\bigg|, \label{App:BesovEquiEq}
\end{equation}
where $K\in C^{\infty}(\Rd)$ such that supp $K\subseteq [0,1)^d$ and whose Fourier transform satisfies 
$|\mathcal{F}[K](\xi)|>0$ for $|\xi|\leq 2$. The norm equivalence~\eqref{App:BesovEquiEq} is well-known in analysis. It follows by an adaptation of the proof of Theorem 1 in~\cite{triebel1988characterizations}.

\begin{proof}[Proof of Proposition~\ref{Res:veri_multi}]
We have to show that the multiscale system $\Phi=\big\{\phi_B\, \big|\, B\in\Omega\big\}$ satisfying Assumption~\ref{ass:multiSystem} also satisfies Assumption~\ref{Res:Gen_Ass} with $\Gamma=\max\{1,d/2\}$. 
For that, note that by part b) of Remark~\ref{remAss1}, we have $n^{\max\{1,d/2\}}\leq \#\Omega_n\leq n^{\max\{1,d/2\}+1}$ for all 
$n\in\mathbb{N}$, so we have $\Gamma=\max\{1,d/2\}$. 
\\

For the inequality in Assumption~\ref{Res:Gen_Ass}, we have to show that there is a constant $C>0$ such that for any $n\in\mathbb{N}$ we have
 \begin{equation}
  \|g\|_{\Besov}\leq \frac{C}{\sqrt{n}}\|g\|_{L^{\infty}(\mathbb{T}^d)}+C\max_{B\in\Omega_n}\bigg|\int_{[0,1)^d}\phi_{B}(z)g(z)\, dz\bigg|
 \end{equation}
 for any $g\in L^{\infty}(\Td)$.
 \\
 For simplicity of the notation, we will denote the cubes in $\Omega_n$ by $\overline{k}+[0,m^{-j})^d$, and the corresponding functions by $\phi_{j,\overline{k}}=m^{jd/2}K(m^j(\cdot-\overline{k}))$, with $j=0,\ldots,J-1$ and 
 $\overline{k}\in\mathcal{D}_{R}$ (see Assumption~\ref{ass:multiSystem} for the definition of this set). 
 With this notation, the claim can be rewritten as
 \begin{equation}
  \|g\|_{\Besov}\leq \frac{C}{m^{Jd/2}}\|g\|_{L^{\infty}(\mathbb{T}^d)}+\max_{0\leq j <J}\max_{\overline{k}\in\mathcal{D}_{R}}\bigg|\int_{[0,1)^d}\phi_{j,\overline{k}}(z)g(z)\, dz\bigg|,\label{App:Claim_Prop_MR}
 \end{equation}
 since $J=\big\lfloor\frac{1}{d}\frac{\log n}{\log m}\big\rfloor$. Finally, without loss of generality we can prove the claim for $m=2$, since the case of general $m>2$ follows analogously.
 \\
 By the characterization of Besov spaces in~\eqref{App:BesovEquiEq}, we have
  \begin{align*}
   \|g\|_{\Besov}&\asymp \sup_{j\in \mathbb{N}_0}\sup_{x\in[0,1)^d}\bigg|\int_{[0,1)^d}\phi_{j,x}(z)g(z)\, dz\bigg|
   \\
   &\leq \sup_{0\leq j < J}\sup_{x\in[0,1)^d}\bigg|\int_{[0,1)^d}\phi_{j,x}(z)g(z)\, dz\bigg|
   \\
   &\hspace{0.5cm} + \sup_{j\geq J}\sup_{x\in[0,1)^d}\bigg|\int_{[0,1)^d}\phi_{j,x}(z)g(z)\, dz\bigg|
  \end{align*}
  for any $J\in\mathbb{N}$. 
 The first term is controlled in Step 1 by
 \begin{equation}
  \sup_{0\leq j < J}\sup_{x\in[0,1)^d}\bigg|\int_{[0,1)^d}\phi_{j,x}(z)g(z)\, dz\bigg|\leq C2^{-Jd/2}\,\|g\|_{L^{\infty}(\mathbb{T}^d)}+ \max_{0\leq j <J}\max_{\overline{k}\in\mathcal{D}_{R}}\bigg|\int_{[0,1)^d}\phi_{j,\overline{k}}(z)g(z)\, dz\bigg|,\label{Mu_app1}
 \end{equation}
 where $\mathcal{D}_{R}$ is the index set of positions. 
 The second term is controlled in Step 2, which gives
 \begin{equation}
  \sup_{j\geq J}\sup_{x\in[0,1)^d}\bigg|\int_{[0,1)^d}\phi_{j,x}(z)g(z)\, dz\bigg|\leq 2^{-Jd/2}\|g\|_{L^{\infty}(\mathbb{T}^d)}\|K\|_{L^1(\Rd)}.\label{mu_app2}
 \end{equation}
 These bounds imply the claim.
 \\
 \textbf{Step 1.} By the definition of the set $\mathcal{D}_{R}$, for any $x\in[0,1)^d$ there is a 
 $\overline{k}\in\mathcal{D}_{R}$ 
 such that $|x-\overline{k}|_{\infty}\leq 2^{-R}$, where $|\cdot|_{\infty}$ denotes the supremum norm in $\mathbb{R}^d$. Hence, 
 for any $j=0,\ldots,J-1$ we have
 \begin{align*}
  \bigg|\int_{[0,1)^d}\phi_{j,x}(z)g(z)\, dz\bigg|&= 2^{jd/2}\bigg|\int_{[0,1)^d}K\big(2^jz\big)g(x+z)\, dz\bigg|
  \\
  &\leq 2^{jd/2}\bigg|\int_{[0,1)^d}K\big(2^jz\big)\big(g(x+z)-g(\overline{k}+z)\big)\, dz\bigg|
  \\
  &\hspace{0.5cm} +2^{jd/2}\bigg|\int_{[0,1)^d}K\big(2^jz\big)g(\overline{k}+z)\, dz\bigg|.
  \end{align*}
 The first term can be bounded as
 \begin{align*}
  2^{jd/2}\bigg|\int_{[0,1)^d}K\big(2^jz\big)&\big(g(x+z)-g(\overline{k}+z)\big)\, dz\bigg|
  \\
  &= 2^{jd/2}\bigg|\int_{[0,1)^d}g(z)\big(K\big(2^j(z-x)\big)-K\big(2^j(z-\overline{k})\big)\big)\, dz\bigg|
  \\\
  &\leq 2^{jd/2}\|g\|_{L^{\infty}(\mathbb{T}^d)}\int_{[0,1)^d}\big|K\big(2^jz\big)-K\big(2^j(z+x-\overline{k})\big)\big|\, dz
  \\
  &=2^{-jd/2}\|g\|_{L^{\infty}(\mathbb{T}^d)}\underbrace{\int_{[0,1)^d}\big|K\big(z\big)-K\big(z+2^j(x-\overline{k})\big)\big|\, dz}_{\leq \big|2^j(x-\overline{k})\big|\|\nabla K\|_{L^1(\mathbb{R}^d)}}
  \\
  &\leq 2^{-jd/2}\|g\|_{L^{\infty}(\mathbb{T}^d)}\underbrace{\big|2^j(x-\overline{k})\big|}_{\leq \sqrt{d}\, 2^{j-R}}\|\nabla K\|_{L^1(\mathbb{R}^d)},
 \end{align*}
 where in the last inequality we used the mean value theorem and the fact that $K$ is smooth. Recall that we have chosen $\overline{k}$ 
 such that $|x-\overline{k}|_{\infty}\leq 2^{-R}$, so that $\big|2^j(x-\overline{k})\big|\leq \sqrt{d}\, \big|2^j(x-\overline{k})\big|_{\infty}\leq \sqrt{d}\, 2^{j-R}$. 
 Since the bound above is uniform in $x\in[0,1)^d$ and $j=0,\ldots,J-1$, we conclude that
 \begin{align*}
  \max_{0\leq j <J}\max_{\overline{k}\in\mathcal{D}_R} 2^{jd/2}&\bigg|\int_{[0,1)^d}K\big(2^jz\big)\big(g(x+z)-g(\overline{k}+z)\big)\, dz\bigg| 
  \\
  &\leq \sqrt{d}\, \max_{0\leq j<J}2^{j(1-d/2)-R}\|g\|_{L^{\infty}(\mathbb{T}^d)}\|\nabla K\|_{L^1(\mathbb{R}^d)}.
 \end{align*}
 The choices $R=J$ if $d=1$ and $R=Jd/2$ if $d\geq 2$ give $\max_{0\leq j<J}2^{j(1-d/2)-R}=2^{-Jd/2}$. Hence, we have
 \begin{align*}
  \max_{0\leq j <J}\sup_{x\in[0,1)^d}\bigg|\int_{[0,1)^d}\phi_{j,x}(z)g(z)\, dz\bigg|&\leq \max_{0\leq j <J}\max_{\overline{k}\in\mathcal{D}_{R}} 2^{jd/2}\bigg|\int_{[0,1)^d}K\big(2^jz\big)\big(g(x+z)-g(\overline{k}+z)\big)\, dz\bigg| 
  \\
  &\hspace{0.5cm}+ \max_{0\leq j <J}\max_{\overline{k}\in\mathcal{D}_{R}}\bigg|\int_{[0,1)^d}K_{j,\overline{k}}(z)g(z)\, dz\bigg|
  \\
  &\leq \sqrt{d}\, 2^{-Jd/2}\,\|g\|_{L^{\infty}(\mathbb{T}^d)}\|\nabla K\|_{L^1(\mathbb{R}^d)}
  \\
  &\hspace{0.5cm}+ \max_{0\leq j <J}\max_{\overline{k}\in\mathcal{D}_{R}}\bigg|\int_{[0,1)^d}K_{j,\overline{k}}(z)g(z)\, dz\bigg|,
 \end{align*}
 which yields~\eqref{Mu_app1}. 
 Summarizing, we have approximated the supremum over $x\in[0,1)^d$ by the supremum over dyadic positions $\overline{k}$ at scale 
 $2^{-R}$.
 
 \textbf{Step 2.} Equation~\eqref{mu_app2} follows by H\"older's inequality, i.e.
 \begin{align*}
  \bigg|\int_{[0,1)^d}\phi_{j,x}(z)g(z)\, dz\bigg|&\leq \|g\|_{L^{\infty}(\mathbb{T}^d)}\int_{[0,1)^d}2^{jd/2}\big|K\big(2^{j}(z-x)\big)\big|\, dz
  \\
  &=2^{-jd/2}\|g\|_{L^{\infty}(\mathbb{T}^d)}\|K\|_{L^1(\Rd)}.
 \end{align*}
 The result follows by taking the supremum over $x\in[0,1)^d$ and over $j\geq J$.
\end{proof}

\subsubsection{Proof of Proposition~\ref{Res:Veri_Curv}}\label{VerificationCurve}

\begin{proof}[Proof of Proposition~\ref{Res:Veri_Curv}]
The inequality in Assumption~\ref{Res:Gen_Ass} follows in both cases (curvelet and shearlet) from the inequality~\eqref{SM:WaveJack} for the wavelet basis (see the proof of Proposition~\ref{Res:Veri_Wavelet} above). Indeed, denoting the elements of $\Phi$ by
\begin{equation*}
 \phi_{\omega}=\begin{cases}
                \psi_{j,k,e} \ &\textup{ if } \omega=(j,k,e)\in\Theta^W,
                \\
                \varphi_{j,\tilde{\theta}} \ &\textup{ if } \omega=(j,\tilde{\theta})\in \Theta,
               \end{cases}
\end{equation*}
we have
\begin{align*}
\|g\|_{\Besov}&\leq C\max_{(j,k,e)\in\Theta^W_n}|\langle g,\psi_{j,k,e}\rangle|+C\frac{\|g\|_{L^{\infty}(\Td)}}{\sqrt{n}}
\\
&\leq C\max_{\omega\in\Theta^W_n\cup\Theta_n}|\langle g,\phi_{\omega}\rangle|+C\frac{\|g\|_{L^{\infty}(\Td)}}{\sqrt{n}},
\end{align*}
where we just enlarge the right-hand side by taking the maximum over a larger index set. 
Concerning the cardinality of $\Omega_n\cup\Theta_n$, by Assumption~\ref{Res:Ass_Curv} we have
\begin{equation*}
\#(\Omega_n\cup\Theta_n)= 2^{d\lfloor\frac{1}{d}\frac{\log n}{\log 2}\rfloor}+2^{d\lfloor\frac{1}{d}\frac{\log n}{\log 2}\rfloor},
\end{equation*}
and hence we have Assumption~\ref{Res:Gen_Ass} with $Q(x)=2x$ and $\Gamma=1$.
\end{proof}

\end{document}